\documentclass[11pt,a4paper]{article}
\usepackage{fullpage}
\usepackage[T1]{fontenc} 
\usepackage{amsmath,amssymb}
\usepackage{amsopn,amsthm}
\usepackage{subcaption}

\usepackage{graphicx}
\usepackage{color}
\usepackage{psfrag}

\usepackage{tikz}
\usepackage{paralist}

\usepackage{mathrsfs}  
\usepackage{bm}
\usepackage{cases}

\usepackage{algpseudocode} 
\usepackage{algorithm}
\floatstyle{plain} \newfloat{myalgo}{tbhp}{mya}

\usepackage[a]{esvect}
\usepackage{float}

{\begin{myalgo}[#1]
    \centering
    \begin{minipage}{0.9\textwidth}
      \begin{algorithm}[H]}%
      {\end{algorithm}
    \end{minipage}
  \end{myalgo}}

\usepackage[colorlinks=true]{hyperref}
\hypersetup{urlcolor=blue, citecolor=blue, linkcolor=blue}

\newtheorem{theorem}{Theorem}[section]
\newtheorem{proposition}[theorem]{Proposition}
\newtheorem{lemma}[theorem]{Lemma}
\newtheorem{corollary}[theorem]{Corollary}
\newtheorem{assumption}[theorem]{Assumption}

\theoremstyle{definition}

\newtheorem{example}[theorem]{Example}

\theoremstyle{remark} \newtheorem{remark}[theorem]{Remark}

\numberwithin{equation}{section}
\numberwithin{figure}{section}
\numberwithin{algorithm}{section}

\definecolor{darkgreen}{rgb}{0.0, 0.5, 0.0}

\newcommand{\Ltwo}{{L^2}}
\newcommand{\Linfty}{{L^{\infty}}}

\newcommand{\LtwoSd}{{L^2(\Sd)}}
\newcommand{\LtwoSone}{{L^2(\Sone)}}

\newcommand{\LtwoBR}{{L^2(B_R(0))}}
\newcommand{\LtwoD}{{L^2(D)}}

\newcommand{\LinftyRd}{{\Linfty(\Rd)}}

\newcommand{\LinftyD}{{\Linfty(D)}}

\newcommand{\HtwolocRd}{H^2_\loc(\Rd)}

\newcommand{\field}[1]{{\mathbb{#1}}}
\newcommand{\C}{\field{C}}
 
\newcommand{\N}{\field{N}}
\newcommand{\R}{\field{R}}

\newcommand{\Dcal}{\mathcal{D}}

\newcommand{\Fcal}{\mathcal{F}}

\newcommand{\Hcal}{\mathcal{H}}

\newcommand{\Ncal}{\mathcal{N}}

\newcommand{\Pcal}{\mathcal{P}}
\newcommand{\Rcal}{\mathcal{R}}

\newcommand{\Tcal}{\mathcal{T}}

\newcommand{\Vcal}{\mathcal{V}}
\newcommand{\Wcal}{\mathcal{W}}

\newcommand{\loc}{{\mathrm{loc}}}

\newcommand{\ol}[1]{\overline{#1}}

\newcommand{\Llra}{\Longleftrightarrow}

\newcommand{\tm}{\subseteq} 
\newcommand{\di}{\partial}
\newcommand{\trans}{{\top}}

\newcommand{\ds}{\, \dif s}

\newcommand{\dx}{\, \dif x}
\newcommand{\dy}{\, \dif y}

\newcommand{\ghat}{\widehat g}

\newcommand{\xhat}{\widehat{x}}

\newcommand{\gtilde}{{\widetilde g}}
\newcommand{\qtilde}{{\widetilde q}}

\newcommand{\vtilde}{{\widetilde v}}

\newcommand{\Ctilde}{{\widetilde C}}

\newcommand{\Vtilde}{{\widetilde V}}

\newcommand{\Vcaltilde}{{\widetilde \Vcal}}

\newcommand{\rhotilde}{{\widetilde\rho}}

\newcommand{\rmi}{\mathrm{i}} 

\newcommand{\eps}{\varepsilon}

\newcommand{\uinfty}{u^\infty}
\newcommand{\vinfty}{v^\infty}

\newcommand{\Sone}{{S^1}}

\newcommand{\Sd}{{S^{d-1}}}
\newcommand{\Rd}{{\R^d}}
\newcommand{\Rtwo}{{\R^2}}

\newcommand{\BR}{{B_R(0)}}

\newcommand{\ui}{u^i}
\newcommand{\us}{u^s}

\newcommand{\ph}{\,\cdot\,}
\newcommand{\skp}[2]{\langle#1,#2\rangle}

\newcommand{\omegad}{\omega_{d-1}}

\newcommand{\qmin}{q_{0,\mathrm{min}}}
\newcommand{\qmax}{q_{0,\mathrm{max}}}

\newcommand{\Vcalperp}{{\Vcal^\perp}}
\newcommand{\Wcalperp}{{\Wcal^\perp}}

\newcommand{\zij}{z_{ij}}
\newcommand{\Ifac}{I_{\mathrm{fac}}}
\newcommand{\Imon}{I_{\mathrm{mon}}}

\newcommand{\eins}{\mathbf{1}}

\DeclareMathAlphabet{\mathbi}{\encodingdefault}{\rmdefault}{\bfdefault}{\itdefault}

\DeclareMathOperator{\dif}{d\!}  

\DeclareMathOperator{\spann}{span}

\DeclareMathOperator{\supp}{supp}

\DeclareMathOperator{\real}{Re}

\DeclareMathOperator{\argmin}{argmin}

\DeclareMathOperator*{\essinf}{ess\,inf}

\begin{document}

\title{Inverse medium scattering for a nonlinear Helmholtz equation} 
\author{Roland Griesmaier\footnote{Institut f\"ur
    Angewandte und Numerische Mathematik, 
    Karlsruher Institut f\"ur Technologie, Englerstr.~2,
    76131 Karlsruhe, Germany ({\tt roland.griesmaier@kit.edu},
    {\tt marvin.knoeller@kit.edu})}\,, 
  Marvin Kn\"oller\footnotemark[1]\,,
  and Rainer Mandel\footnote{Institut f\"ur Analysis, 
    Karlsruher Institut f\"ur Technologie, Englerstr.~2,
    76131 Karlsruhe, Germany ({\tt rainer.mandel@kit.edu}).}
}
\date{\today}

\maketitle

\begin{abstract}
  We discuss a time-harmonic inverse scattering problem for a
  nonlinear Helmholtz equation with compactly supported inhomogeneous
  scattering objects that are described by a nonlinear refractive
  index in unbounded free space. 
  Assuming the knowledge of a nonlinear far field operator, which maps
  Herglotz incident waves to the far field patterns
  of corresponding solutions of the nonlinear scattering problem, we
  show that the nonlinear index of refraction is uniquely determined. 
  We also generalize two reconstruction methods, a factorization
  method and a monotonicity method, to recover the support of such
  nonlinear scattering objects. 
  Numerical results illustrate our theoretical findings. 
\end{abstract}

{\small\noindent
  Mathematics subject classifications (MSC2010): 35R30, (65N21)
  \\\noindent
  Short title: Nonlinear inverse medium scattering
}

\section{Introduction}
\label{sec:Introduction}
The linear Helmholtz equation is used to model the propagation
of sound waves or electromagnetic waves of small amplitude in
inhomogeneous isotropic media in the time-harmonic regime (see, e.g.,
\cite{ColKre19}). 
However, if the magnitudes are large, then intensity-dependent
material laws might be required, and nonlinear Helmholtz equations
are often more appropriate.
A prominent example are Kerr-type nonlinear media (see,
e.g.,~\cite{Boy08,MolNew04} for the physical background). 
Optical Kerr effects are studied in various applications from laser
optics (see, e.g., \cite{AdaChaPay89,Che16}) both from a theoretical
and applied point of view. 
In this theoretical study we consider an inverse medium scattering
problem for a class of nonlinear Helmholtz equations that covers for
instance generalized Kerr-type nonlinear media of
arbitrary order.

To begin with, we discuss the well-posedness of the direct scattering 
problem. 
We consider compactly supported scatterers that are described by a
nonlinear refractive index, which we basically assume to be well
approximated by a linear refractive index at low intensities. 
Rewriting the scattering problem in terms of a nonlinear
Lippmann-Schwinger equation we use a contraction argument together
with resolvent estimates for the linearized problem to establish the 
existence and uniqueness of solutions for incident waves that are
sufficiently small relative to the size of the nonlinearity.
Here it is important to note that the parameters in nonlinear
material laws are usually extremely small (see, e.g.,
\cite[p.~212]{Boy08}), which means that this assumption does not
rule out incident fields of rather large intensity.
As a byproduct we also give a priori estimates for the solution of
the nonlinear scattering problem as well as estimates for the
linearization error, which are instrumental for the rest of the work.
The main reason for considering incident waves that are small relative
to the size of the nonlinearity here is that we later use
linearization techniques to solve the corresponding inverse problem. 
However, we note that a more general existence result for
the direct scattering problem that avoids any smallness assumption on
the incident field has recently been established in \cite{CheEveWet21}
(see also \cite{Gut04,Man19}).

We define a nonlinear far field operator that maps
densities of Herglotz incident fields to the far field patterns
of the corresponding solutions of the direct scattering problem.
In the linear case such far field operators are used to describe the
scattering process for infinitely many incident fields, and their
properties have been widely studied (see, e.g., \cite{ColKre19}). 
Similar to \cite{Lec11} (see also \cite{KirGri08} for the linear case)
we derive a factorization of this operator into three simpler
operators.
Here it is important to note that only the second operator in this
factorization is nonlinear.
We derive estimates for the corresponding linearization error. 

Restricting the discussion to a class of generalized Kerr-type
nonlinearities of arbitrary order, we then turn to the associated
inverse scattering problem. 
We show that the knowledge of the nonlinear far field operator
uniquely determines the nonlinear refractive index.
This generalizes earlier results for the inverse medium scattering
problem for nonlinear Helmholtz equations from \cite{Fur20,Jal04}.
In comparison to these works we consider a less regular and more
general class of nonlinear refractive indices.
Our proof relies on linearization to determine the terms in the
generalized Kerr-type nonlinearity recursively, and it uses the
classical uniqueness result for the corresponding linear inverse
medium scattering problem (see, e.g.,
\cite{Buk08,Nac88,Nov88,Ram88}). 
Recently, a uniqueness proof that avoids the use of the linear result
has been established for a more regular class of power-type
nonlinearities than considered here in
\cite{FeiOks20,LasLiiLinSal21,HarLin22}. 
Earlier uniqueness results for semilinear elliptic inverse problems
have, e.g., been obtained in \cite{ImaYam13,IsaNac95,IsaSyl94,Sun10}.
Furthermore, inverse scattering problems for nonlinear Schr\"odinger
equations, which are closely related to the nonlinear Helmholtz
equations considered in this work, have been studied using different
techniques than those applied in this work in
\cite{HarSer14,Ser08,Ser12,SerHar08,SerHarFot12,SerSan10}. 

We also generalize two popular methods for shape reconstruction for
inverse scattering problems, the factorization method and the
monotonicity method, to the nonlinear scattering problem.
A related factorization method has been discussed in~\cite{Lec11} for
a class of weakly scattering objects and for scattering objects with
small nonlinearity of linear growth.
In comparison to this work we consider a larger class of
nonlinearities without any smallness assumption on the nonlinearity,
but on the other hand we assume that the incident fields are
sufficiently small relative to the size of the nonlinearity. 
For linear scattering problems the factorization method has originally
been developed in \cite{Kir98,Kir99,Kir00} (see also \cite{ColKir96}
and the monographs \cite{CakCol14,KirGri08}). 
Using estimates for the linearization error we show that the
inf-criterion from \cite{Kir00} can be extended to the nonlinear
case considered in this work.
However, since the far field operator is nonlinear, the efficient
numerical implementation of this criterion using spectral theory that
is used for the linear scattering problem no longer applies.
Instead we have to solve a nonlinear constrained optimization problem
for each sampling point to decide whether it belongs to the support of
the nonlinear scatterer or not. 
This leads to a numerical scheme that is considerably more time
consuming than the traditional scheme for the linear case.

The situation is similar for the nonlinear monotonicity method.
For linear scattering problems monotonicity based reconstruction
methods have been proposed in
\cite{AlbGri20,GriHar18,HarPohSal19b,HarPohSal19a}.
Using linearization techniques we show that the method can be extended
to the nonlinear case considered in this work.
Again the tools from spectral theory that have been used for the
numerical implementation of the monotonicity criteria in
\cite{AlbGri20,GriHar18} are not available for the nonlinear
scattering problem.
However, we show that there is a close connection between the
nonlinear monotonicity based shape characterization and the
inf-criterion for the nonlinear factorization method, which we exploit
to implement the nonlinear monotoncity based reconstruction method in
terms of a similar constrained optimization problem as for the
nonlinear factorization method.

We consider a numerical example with a scattering object that is
described by a third-order nonlinear refractive index using optical
coefficients for glass from \cite{Boy08}.
Since the nonlinear part of the refractive index is extremely small,
we work with incident fields of very high intensity such that there
is a significant nonlinear contribution in the scattered field. 
The forward solver, which is based on the same fixed point
iteration for the nonlinear Lippmann-Schwinger equation that we use
to analyze the direct scattering problem, as well as the
reconstruction methods work well.
This suggests that the smallness assumptions on the intensity of the
incident fields that we have to make in our theoretical results is not
too restrictive. 

The article is organized as follows.
In Section~\ref{sec:Setting} we introduce the nonlinear scattering
problem, and we establish existence and uniqueness of solutions for
the direct scattering problem.
In Section~\ref{sec:InverseProblem} we turn to the inverse scattering
problem to recover the nonlinear refractive index from observations of
the corresponding nonlinear far field operator.
Focusing on a class of generalized Kerr-type nonlinearities, we show
that this inverse problem has a unique solution. 
In Sections~\ref{sec:FactorizationMethod} and
\ref{sec:MonotonicityMethod} we derive and analyze a nonlinear
factorization method and a nonlinear monotonicity method for
reconstructing the support of nonlinear scatterers. 
In Section~\ref{sec:NumericalExamples} we provide numerical examples.

\section{The nonlinear scattering problem}
\label{sec:Setting}
The nonlinear wave equation 
\begin{equation*}
  \frac{\di^2\psi}{\di t^2}(t,x) - \Delta \psi(t,x) 
  \,=\, h(x,\psi(t,x)) \,, 
  \qquad (t,x) \in \R\times\Rd \,,
\end{equation*}
is used to model the interaction of acoustic or electromagnetic waves 
with a compactly supported inhomogeneous penetrable scattering object
with nonlinear response in $d$-dimensional free space for $d=2,3$.  
In the following we restrict the discussion to nonlinearities of the
form 
\begin{equation*}
  h(x,\psi(t,x)) 
  \,=\, k^2 q(x,|\psi(t,x)|)\psi(t,x) \,, 
  \qquad (t,x) \in \R\times\Rd \,,
\end{equation*}
where $q: \Rd\times\R\to\R$ is real-valued.
Specifying a \emph{wave number} $k>0$, the time-periodic ansatz 
\begin{equation*}
  \psi(x,t) 
  \,=\, e^{-\rmi k t} u(x) \,, \qquad (x,t) \in \Rd\times\R \,,
\end{equation*}
gives the nonlinear Helmholtz equation
\begin{equation*}
  \Delta u + k^2 u 
  \,=\, - k^2 q(x,|u|) u \,, \qquad x\in\Rd \,.
\end{equation*}
Denoting by $n^2 := 1+q$ the associated nonlinear 
\emph{refractive index}, we make the following general assumptions
throughout this work. 

\begin{assumption}
  \label{ass:Index1}
  The nonlinear contrast function $q\in\Linfty(\Rd\times\R)$ shall satisfy 
  \begin{itemize}
  \item[(i)] $\supp(q) \tm \overline D\times\R$ for some bounded open
    set $D\subset\Rd$, 
  \item[(ii)] $q(x,0)=0$ for a.e.\ $x\in\Rd$,
  \item[(iii)] and there exist $q_0\in\LinftyRd$ with 
    $\essinf q_0>-1$ and $\supp(q_0)\tm \ol{D}$,
    and a parameter~$\alpha>0$ such that for
    any $z_1,z_2\in\C$ with $|z_1|,|z_2|\leq 1$, 
    \begin{equation}
      \label{eq:Assumption3}
      \bigl\|q(\ph,|z_1|)z_1-q(\ph,|z_2|)z_2-q_0(z_1-z_2)\bigr\|_{\LinftyRd}
      \,\leq\, C_q (|z_1|^\alpha + |z_2|^\alpha) |z_1-z_2| \,.
    \end{equation}
  \end{itemize}
\end{assumption}

For later reference we note that \eqref{eq:Assumption3} implies 
\begin{equation}
  \label{eq:Assumption3a}
  \bigl\|q(\ph,|z|)z-q_0 z\bigr\|_{\LinftyRd} 
  \,\leq\, C_q |z|^{1+\alpha} \qquad 
  \text{for any } z\in\C \,,\; |z|\leq 1 \,.
\end{equation}

\begin{example}
  \label{exa:generalizedKerr}
  An example for a nonlinear material law that satisfies 
  Assumption~\ref{ass:Index1} is the
  \emph{generalized Kerr-type material law} 
  \begin{equation}
    \label{eq:generalizedKerr}
    q(x,|z|) 
    \,=\, q_0(x) + \sum_{l=1}^L q_l(x)|z|^{\alpha_l} 
    \qquad x\in\Rd \,,\; z\in\C \,,
  \end{equation} 
  for $q_0,\ldots,q_L\in \LinftyRd$ with support in $\ol{D}$, where
  the lowest order term satisfies $\essinf q_0>-1$, and the 
  exponents fulfill $0<\alpha_1<\cdots<\alpha_L<\infty$. 
  In this case condition~\eqref{eq:Assumption3} is satisfied for 
  $\alpha=\alpha_1$ and $C_q=\sum_{l=1}^L\|q_l\|_\LinftyD$.
  For the special case when~$L=1$ and $\alpha_1=2$ this gives the
  well-known Kerr nonlinearity (see, e.g., \cite{Boy08,MolNew04}).
  \hfill$\lozenge$ 
\end{example}

We suppose that the wave motion is caused by an \emph{incident field}
$\ui$ satisfying the linear Helmholtz equation 
\begin{subequations}
  \label{eq:ScatteringProblem}
  \begin{equation}
    \label{eq:uIncident}
    \Delta \ui + k^2 \ui \,=\, 0 \qquad \text{in } \Rd \,.
  \end{equation}
  The scattering problem that we consider consists in determining the
  \emph{total field} $u=\ui+\us$ such that 
  \begin{equation}
    \label{eq:uTotal}
    \Delta u + k^2 n^2(\ph,|u|) u
    \,=\, 0 \qquad \text{in } \Rd \,,
  \end{equation}
  where the \emph{scattered field} $\us$ satisfies the
  \emph{Sommerfeld radiation condition} 
  \begin{equation}
    \label{eq:Sommerfeld}
    \lim_{r\to\infty} r^{\frac{d-1}{2}} \Bigl(
    \frac{\di\us}{\di r}(x) - \rmi k \us(x) \Bigr)
    \,=\, 0 \,, \qquad r=|x| \,,
  \end{equation}
\end{subequations}
uniformly with respect to all directions $x/|x|\in\Sd$.

\begin{remark}
  Throughout this work (nonlinear) Helmholtz equations are to be
  understood in the strong sense.
  For instance, $u\in\HtwolocRd$ is a solution to \eqref{eq:uTotal}
  if and only if it satisfies the equation weakly almost everywhere in
  $\Rd$. 
  Elliptic regularity results show that $\ui$ is smooth throughout
  $\Rd$, and that $u$ and thus also $\us$ are smooth in 
  $\Rd\setminus\ol{D}$. 
  In particular the radiation condition \eqref{eq:Sommerfeld} is
  well-defined. 
  As usual we call a solution to a (nonlinear) Helmholtz equation on
  an unbounded domain that satisfies the Sommerfeld radiation
  condition a \emph{radiating solution}.~\hfill$\lozenge$
\end{remark}

Next we show that the scattering problem \eqref{eq:ScatteringProblem}
is equivalent to the problem of solving the 
\emph{nonlinear Lippmann-Schwinger equation}
\begin{equation}
  \label{eq:LippmannSchwinger}
  u(x) 
  \,=\, \ui(x) + k^2 \int_D \Phi_k(x-y) q(y,|u(y)|) u(y) \dy \,, 
  \qquad x\in D\,,
\end{equation}
in $\LinftyD$.
Here $\Phi_k$ is the outgoing free space fundamental solution to the
Helmholtz equation, i.e., for $x,y\in\Rd$, $x\not=y$, we have 
$\Phi_k(x-y)=(\rmi/4)\, H^{(1)}_0(k|x-y|)$ if $d=2$ and 
$\Phi_k(x-y)=e^{\rmi k|x-y|}/(4\pi|x-y|)$ if $d=3$.
The arguments that we use to prove this are the same as in the linear
case (see, e.g., \cite[Thm.~7.12]{Kir21}). 

\begin{lemma}
  \label{lmm:LippmannSchwinger}
  If $u\in\HtwolocRd$ is a solution of \eqref{eq:ScatteringProblem},
  then $u|_{D}$ is a solution of \eqref{eq:LippmannSchwinger}.
  Conversely, if $u\in\LinftyD$ is a solution of
  \eqref{eq:LippmannSchwinger} then $u$ can be extended to a solution
  $u\in\HtwolocRd$ of \eqref{eq:ScatteringProblem}.
\end{lemma}

\begin{proof}
  Let $u\in\HtwolocRd$ be a solution of
  \eqref{eq:ScatteringProblem}. 
  Then $q(\ph,|u|) u|_{D}\in \LinftyD$, and the volume potential 
  $v := k^2 \Phi_k\ast (q(\ph,|u|) u) \in \HtwolocRd$ is a radiating
  solution of
  \begin{equation}
    \label{eq:ProofLS1}
    \Delta v + k^2 v 
    \,=\, - k^2 q(\ph,|u|) u \qquad \text{in } \Rd 
  \end{equation}
  (see, e.g., \cite[Thm.~7.11]{Kir21}).  
  Accordingly, $\us-v$ is a radiating solution of 
  $\Delta(\us-v)+k^2(\us-v)=0$ in~$\Rd$. 
  Thus $v = \us$ (see, e.g., \cite[p.~24]{ColKre19}), which proves the
  first part.

  Conversely, let $u\in \LinftyD$ be a solution of 
  \eqref{eq:LippmannSchwinger}.
  Defining $v := k^2 \Phi_k\ast (q(\ph,|u|) u)$ in $\Rd$, we find that
  $u=\ui+v$ in $D$. 
  Moreover, $v \in \HtwolocRd$ satisfies \eqref{eq:ProofLS1}, and if
  we extend $u$ by $\ui+v$ to all of $\Rd$, then $u$ solves
  \eqref{eq:ScatteringProblem}. 
\end{proof}

In the following we consider this problem for more general source
terms and study radiating solutions $v\in\HtwolocRd$ of 
\begin{equation}
  \label{eq:GeneralSP}
  \Delta v + k^2 v
  \,=\, -k^2 q(\ph,|v+f|) (v+f) 
  \qquad \text{in } \Rd \,,
\end{equation}
where $f\in \LinftyD$.
In this situation, $f$ represents the incident field and $v$ the
corresponding scattered field. 
As in Lemma~\ref{lmm:LippmannSchwinger} we find that this is
equivalent to the problem of solving the nonlinear integral equation 
\begin{equation}
  \label{eq:GeneralLS}
  v(x) 
  \,=\,  k^2 \int_{D}  \Phi_k(x-y) q(y,|v(y)+f(y)|) (v(y)+f(y)) \dy \,, 
  \qquad x\in D \,,
\end{equation}
in $\LinftyD$. 

\begin{remark}
  In the linear case, i.e., when $q=q_0$, the scattering problem
  \eqref{eq:GeneralSP} reduces to 
  \begin{equation}
    \label{eq:LinearSP}
    \Delta v_0 + k^2 v_0
    \,=\, -k^2 q_0 (v_0+f) 
    \qquad \text{in } \Rd \,,
  \end{equation}
  and the corresponding linear Lippmann-Schwinger equation reads
  \begin{equation}
    \label{eq:LinearLS}
    v_0(x) 
    \,=\,  k^2 \int_{D} \Phi_k(x-y) q_0(y) (v_0(y)+f(y)) \dy \,, 
    \qquad x\in D\,.
  \end{equation}
  We note that $I-k^2\Phi_k\ast (q_0\ph)$ is an isomorphism
  on~$\LtwoD$ (see \cite[Thm.~7.13]{Kir21} for the corresponding
  result in the case when $D$ is a ball $\BR$) as well as
  on~$\LinftyD$. 
  For the latter we recall that $k^2\Phi_k\ast (q_0\ph)$ maps
  $\LinftyD$ into $H^2(\BR)$ for $\BR$ containing $D$, which embeds
  continuously into $\LinftyD$. 
  In particular we have 
  \begin{subequations}
    \label{eq:EstLS}
    \begin{align}
      \bigl\| \bigl( I-k^2\Phi_k\ast (q_0\ph) \bigr)^{-1} g \bigr\|_\LtwoD
      &\,\leq\, C_{LS,2} \| g \|_\LtwoD \,,
      &&g\in \LtwoD \,, 
        \label{eq:EstLSLtwo} \\
      \bigl\| \bigl( I-k^2\Phi_k\ast (q_0\ph) \bigr)^{-1} g \bigr\|_\LinftyD
      &\,\leq\, C_{LS,\infty} \| g \|_\LinftyD \,,
      &&g\in \LinftyD \,. 
        \label{eq:EstLSLinfty}
    \end{align}
  \end{subequations}
  Accordingly, the unique solution $v_0$ of \eqref{eq:LinearLS} is
  given by
  \begin{equation}
    \label{eq:DefV0}
    v_0 
    \,=\, \bigl( I-k^2\Phi_k\ast (q_0\ph) \bigr)^{-1}
    \bigl( k^2 \Phi_k\ast (q_0 f) \bigr) 
    \qquad \text{in } D\,,
  \end{equation}
  and we denote by $V_0$ the linear operator that maps $f$ to $v_0$. 
  The solution $v_0$ can be extended by the right hand side of
  \eqref{eq:LinearLS} to a radiating solution of \eqref{eq:LinearSP} 
  in all of~$\Rd$, which we also denote by $v_0=V_0f$. 
  For later reference we note that   
  \eqref{eq:EstLS} implies
  \begin{subequations}
    \label{eq:V0Estimate}
    \begin{align}
      \|V_0f\|_\LtwoD
      &\,\leq\, C_{V_0,2} \|f\|_\LtwoD \,,
      &&f\in \LtwoD \,,
         \label{eq:V0EstimateLtwo} \\
      \|V_0f\|_\LinftyD
      &\,\leq\, C_{V_0,\infty} \|f\|_\LinftyD \,,
      &&f\in \LinftyD \,,       
         \label{eq:V0EstimateLinfty}
    \end{align}
  \end{subequations}
  where
  $C_{V_0,\infty}=k^2C_{LS,\infty}\|\Phi_k\|_{L^1(B_{2R}(0))}\|q_0\|_\LinftyD$
  and
  $C_{V_0,2}=k^2C_{LS,2}\|\Phi_k\|_{L^1(B_{2R}(0))}\|q_0\|_\LinftyD$. 
  Here and in the following $R>0$ is chosen such that~$D\tm\BR$.
  \hfill$\lozenge$
\end{remark}

In Proposition~\ref{pro:Wellposedness} below we establish
well-posedness of \eqref{eq:GeneralSP}. 
Writing
\begin{equation*}
  U_\delta 
  \,:=\, \bigl\{ v\in \LinftyD \,\big|\, 
  \|v\|_\LinftyD \leq \delta \bigr\} \,,
  \qquad \delta>0 \,,
\end{equation*}
we show that for any $f\in U_{\delta}$ with $\delta>0$ sufficiently
small there exists a unique solution $v$ of~\eqref{eq:GeneralLS} in
$\LinftyD$ such that the difference $w:=v-v_0$ with $v_0$ from
\eqref{eq:DefV0} satisfies $w\in U_{\delta}$. 
We call this~$v$ the unique small solution of \eqref{eq:GeneralLS}. 
Denoting by $V$ the nonlinear operator that maps $f$ to~$v$, we shall
see that $V$ is Fr\'{e}chet-differentiable at zero and $V'(0)=V_0$. 
The mere existence of such an operator is well-known, see for
instance, \cite[Thm.~1.2]{CheEveWet21}, \cite[Thm.~1]{Gut04}, or
\cite[Thm.~1]{Man19}. 

\begin{proposition}
  \label{pro:Wellposedness}
  Suppose that Assumption~\ref{ass:Index1} is satisfied. 
  There exists $\delta > 0$ such that for any given $f \in U_{\delta}$
  the nonlinear integral equation \eqref{eq:GeneralLS} has a unique
  solution ${v = V(f) \in \LinftyD}$ satisfying $v-V_0f \in U_{\delta}$, and
  there exists a constant\footnote{Throughout $C$ denotes a generic
    constant, the values of which might change from line to line.}
  $C>0$ such that, for all such~$f$,
  \begin{subequations}
    \label{eq:VEstimates}
    \begin{align}
      \|V(f) \|_\LinftyD 
      &\,\leq\, C \|f\|_\LinftyD \,, 
        \label{eq:VStabilityLinfty}\\
      \|V(f) \|_\LtwoD 
      &\,\leq\, C \|f\|_\LtwoD \,, 
        \label{eq:VStabilityLtwo}\\
      \|V(f)- V_0f \|_\LinftyD 
      &\,\leq\, C \|f\|_\LinftyD^{1+\alpha} \,, 
        \label{eq:VHoelderLinfty}\\ 
      \|V(f)- V_0f \|_\LtwoD 
      &\,\leq\, C \|f\|_\LinftyD^{\alpha} \|f\|_\LtwoD \,. 
        \label{eq:VHoelderLtwo}
    \end{align}
  \end{subequations}
\end{proposition}

\begin{remark}
  \label{rem:Wellposedness}
  The proof of Proposition~\ref{pro:Wellposedness} below shows that
  the upper bound $\delta>0$ has to be such that the product
  $C_q\delta>0$, where $C_q$ is the upper bound on the nonlinearity
  from Assumption~\ref{ass:Index1}, is sufficiently small. 
  This means that there is a tradeoff between the size of the
  nonlinearity and the intensity of the incident fields and
  scattered fields that are covered by this well-posedness result. 
  \hfill$\lozenge$
\end{remark}

\begin{proof}[Proof of Proposition~\ref{pro:Wellposedness}]
  For any given $f\in \Linfty(D)$ let $v_0:=V_0f\in\LinftyD$ as in
  \eqref{eq:DefV0}. 
  Then, $v\in\LinftyD$ solves~\eqref{eq:GeneralLS} if and only if
  $w := v-v_0$ satisfies
  \begin{equation*}
    w - k^2\Phi_k \ast (q_0w)
    \,=\, 
    k^2\Phi_k\ast \bigl( q_N(\ph,|w+v_0+f|)(w+v_0+f) \bigr)
    \qquad \text{in } D  \,,
  \end{equation*}
  where $q_N := q-q_0$ denotes the nonlinear part of the contrast
  function. 
  This is equivalent to $w$ being a fixed point of the nonlinear map 
  $G: \LinftyD \to \LinftyD$, 
  \begin{equation}
    \label{eq:DefG}
    G(w) 
    \,:=\, \bigl(I-k^2\Phi_k \ast (q_0\ph)\bigr)^{-1}
    \Bigl(k^2\Phi_k\ast \bigl( q_N(\ph,|w+v_0+f|)(w+v_0+f) \bigr) \Bigr) \,.
  \end{equation}
  Using \eqref{eq:EstLSLinfty}, Young's inequality, 
  \eqref{eq:Assumption3a}, and \eqref{eq:V0EstimateLinfty} we have for
  any $f\in U_{\delta}$ and $w\in U_{\delta}$ that
  \begin{equation*}
    \begin{split}
      \|G(w)\|_\LinftyD
      &\,\leq\, C_{LS,\infty} \bigl\|
      k^2\Phi_k\ast
      \bigl(q_N(\ph,|w+v_0+f|)(w+v_0+f)\bigr) \bigr\|_\LinftyD \\
      &\,\leq\, k^2 C_{LS,\infty} \|\Phi_k\|_{L^1(B_{2R}(0))} 
      \bigl\|q_N(\ph,|w+v_0+f|)(w+v_0+f)\bigr\|_\LinftyD \\
      &\,\leq\, k^2 C_{LS,\infty} \|\Phi_k\|_{L^1(B_{2R}(0))}
      C_q \|w+v_0+f\|_\LinftyD^{1+\alpha} \\
      &\,\leq\, k^2 C_{LS,\infty} \|\Phi_k\|_{L^1(B_{2R}(0))}
      C_q \bigl( \delta + (C_{V_0,\infty}+1)\delta \bigr)^{1+\alpha} \,. 
    \end{split}
  \end{equation*}
  Here, $R>0$ was chosen such that $D\tm\BR$.
  Similarly, applying \eqref{eq:Assumption3} we obtain for
  any~$f\in U_{\delta}$ and $w_1,w_2\in U_{\delta}$ that 
  \begin{equation*}
    \begin{split}
      &\|G(w_1)-G(w_2)\|_\LinftyD\\
      &\,\leq\, k^2 C_{LS,\infty} \|\Phi_k\|_{L^1(B_{2R}(0))}
      C_q \bigl(\|w_1+v_0+f\|_\LinftyD^\alpha
      + \|w_2+v_0+f\|_\LinftyD^\alpha \bigr)
      \|w_1-w_2\|_\LinftyD \\
      &\,\leq\, k^2 C_{LS,\infty} \|\Phi_k\|_{L^1(B_{2R}(0))}
      C_q\, 2 \bigl(\delta + (C_{V_0,\infty}+1)\delta \bigr)^\alpha
      \|w_1-w_2\|_\LinftyD\,. 
    \end{split}
  \end{equation*}
  Choosing $\delta > 0$ such that $C_q\delta > 0$ is sufficiently
  small, we find that 
  \begin{equation*}
    \|G(w)\|_\LinftyD \,\leq\, \delta \,,\qquad
    \|G(w_1)-G(w_2)\|_\LinftyD
    \,\leq\, \frac{1}{2}\|w_1-w_2\|_\LinftyD \,.  
  \end{equation*}
  So $G:U_{\delta}\to  U_{\delta} $ is a contraction, and Banach's
  fixed point theorem yields the existence of a uniquely determined
  fixed point $w\in U_{\delta}$ of $G$ such that $v=V(f):=w+V_0f$
  solves \eqref{eq:GeneralLS}. 
  
  It remains to show~\eqref{eq:VStabilityLinfty}--\eqref{eq:VHoelderLtwo}. 
  This follows from \eqref{eq:DefG}, \eqref{eq:EstLSLinfty}, Young's
  inequality, and \eqref{eq:Assumption3a} because 
  \begin{equation}
    \label{eq:ProofWellposedness7}
    \begin{split}
      &\|V(f)-V_0f\|_\LinftyD
      \,=\, \|G(V(f)-V_0f)\|_\LinftyD \\
      &\,=\, \bigl\| \bigl(I-k^2\Phi_k \ast (q_0\ph)\bigr)^{-1}
      \Bigl(k^2\Phi_k\ast \bigl( q_N(\ph,|V(f)+f|)(V(f)+f) \bigr) \Bigr)
      \bigr\|_\LinftyD \\
      &\,\leq\, k^2 C_{LS,\infty} \|\Phi_k\|_{L^1(B_{2R}(0))}
      \bigl\|q_N(\ph,|V(f)+f|)(V(f)+f)\bigr\|_\LinftyD  \\
      &\,\leq\, k^2 C_{LS,\infty} \|\Phi_k\|_{L^1(B_{2R}(0))}
      C_q \|V(f)+f\|_\LinftyD^{1+\alpha} \\
      &\,\leq\, k^2 C_{LS,\infty} \|\Phi_k\|_{L^1(B_{2R}(0))}
      C_q \bigl( \|V(f)-V_0f\|_\LinftyD+\|V_0f+f\|_\LinftyD \bigr)^\alpha
      \|V(f)+f\|_\LinftyD \,.
    \end{split}
  \end{equation}
  Hence, \eqref{eq:V0EstimateLinfty} yields 
  \begin{equation*}
    \begin{split}
      \|V(f)\|_\LinftyD 
      &\,\leq\, C_{V_0,\infty}\|f\|_\LinftyD\\
      &\phantom{\,\leq\,}
      + k^2 C_{LS,\infty} \|\Phi_k\|_{L^1(B_{2R}(0))}
      C_q  \bigl(\delta+(C_{V_0,\infty}+1)\delta)^\alpha
      \bigl(\|V(f)\|_\LinftyD +\|f\|_\LinftyD\bigr) \,.
    \end{split}
  \end{equation*}
  Given that $C_q\delta> 0$ is sufficiently small as
  in the first part of the proof we thus
  obtain~\eqref{eq:VStabilityLinfty}.  
  Therewith, \eqref{eq:ProofWellposedness7}
  shows~\eqref{eq:VHoelderLinfty}. 
  Finally, using \eqref{eq:DefG}, \eqref{eq:EstLSLtwo}, Young's
  inequality, and \eqref{eq:Assumption3a} we get 
  \begin{equation*}
    \begin{split}
      &\|V(f)-V_0f\|_\LtwoD\\
      &\,=\, \|G(V(f)-V_0f)\|_\LtwoD\\
      &\,=\, \bigl\| \bigl(I-k^2\Phi_k \ast (q_0\ph)\bigr)^{-1}
      \Bigl(k^2\Phi_k\ast \bigl( q_N(\ph,|V(f)+f|)(V(f)+f) \bigr) \Bigr)
      \bigr\|_\LtwoD \\
      &\,\leq\, k^2 C_{LS,2} \|\Phi_k\|_{L^1(B_{2R}(0))}
      \bigl\|q_N(\ph,|V(f)+f|)(V(f)+f)\bigr\|_\LtwoD  \\
      &\,\leq\, k^2 C_{LS,2} \|\Phi_k\|_{L^1(B_{2R}(0))}
      C_q \||V(f)+f|^{1+\alpha}\|_\LtwoD \\
      &\,\leq\, k^2 C_{LS,2} \|\Phi_k\|_{L^1(B_{2R}(0))}
      C_q \|V(f)+f\|_\LinftyD^\alpha \|V(f)+f\|_\LtwoD \,.  
    \end{split}
  \end{equation*}
  Proceeding as before this implies \eqref{eq:VStabilityLtwo} when 
  $C_q\delta>0$ is sufficiently small, and thus
  also~\eqref{eq:VHoelderLtwo}.  
\end{proof}

After extending the right hand side of \eqref{eq:GeneralLS} to all
of~$\Rd$, Proposition~\ref{pro:Wellposedness} guarantees the existence
of a unique small radiating solution of the generalized scattering
problem~\eqref{eq:GeneralSP} for any $f\in\LinftyD$ that is
sufficiently small. 
We denote this extension by $v=V(f)$ as well. 
In particular, Proposition~\ref{pro:Wellposedness} tells us that for
all $\LinftyD$-small incoming waves $\ui$ we have a unique
small solution $u = \ui + V(\ui|_{D})$ of the nonlinear forward
problem~\eqref{eq:ScatteringProblem}. 
Here small means
that~$\|V(\ui|_{D})-V_0(\ui|_{D})\|_\LinftyD \leq \delta$ with
$\delta>0$ from Proposition~\ref{pro:Wellposedness}.
Substituting the far field asymptotics of the fundamental solution
(see, e.g., \cite[p.~24 and p.~89]{ColKre19}) into the extension of
the integral representation \eqref{eq:GeneralLS} to all of $\Rd$, we
obtain the following result. 

\begin{proposition}
  \label{pro:Farfield}
  Suppose that Assumption~\ref{ass:Index1} is satisfied, let 
  $\delta>0$ be as in Proposition~\ref{pro:Wellposedness}, and let
  $f\in U_{\delta}$.  
  Then the extension of the unique solution 
  $v = V(f) \in U_{\delta}$ of \eqref{eq:GeneralLS} to all of $\Rd$,
  has the asymptotic behavior 
  \begin{equation*}
    v(x) 
    \,=\, C_d  e^{\rmi k|x|}|x|^{\frac{1-d}{2}} \vinfty(\xhat) 
    + O(|x|^{-\frac{d+1}{2}}) \,, \qquad |x|\to\infty \,,
  \end{equation*}
  uniformly in all directions $\xhat:=x/|x|\in\Sd$, where
  \begin{equation*}
    C_d \,=\, e^{\rmi\pi/4} /\sqrt{8\pi k} \quad\text{if } d=2 
    \qquad\text{and}\qquad 
    C_d \,=\, 1/(4\pi) \quad\text{if } d=3 \,.
  \end{equation*}
  The \emph{far field pattern} $\vinfty = (V(f))^\infty \in \LtwoSd$
  is given by 
  \begin{equation}
    \label{eq:Farfield}
    \vinfty(\xhat) 
    \,=\, k^2 \int_{D} q\bigl(y,|v(y)+f(y)|\bigr) 
    \bigl(v(y)+f(y)\bigr) e^{-\rmi k \xhat\cdot y} \dy \,,
    \qquad \xhat\in\Sd \,.
  \end{equation}
\end{proposition}

In the following we will we restrict the discussion to incident fields
that are superpositions of plane waves. 
We define the \emph{Herglotz operator} $H:\LtwoSd \to \LtwoD$,
\begin{equation}
  \label{eq:HerglotzOperator}
  (H \psi)(x) 
  \,:=\, \int_\Sd \psi(\theta) e^{\rmi k x\cdot\theta} \ds(\theta) \,,
  \qquad x\in{D} \,,  
\end{equation}
and we note that  its adjoint $H^*:\LtwoD \to \LtwoSd$ satisfies
\begin{equation}
  \label{eq:HerglotzAdjoint}
  (H^*\phi)(\xhat) 
  \,=\, \int_{D} \phi(y) e^{-\rmi k \xhat\cdot y} \dy \,,
  \qquad \xhat\in\Sd \,.
\end{equation}
The operators $H$ and $H^*$ are compact. 
Observing that 
\begin{equation}
  \label{eq:BoundHerglotzLinfty}
  \|H\psi\|_\LinftyD 
  \,\leq\, \omegad^{1/2}\|\psi\|_\LtwoSd \,,
\end{equation}
where $\omegad$ denotes the area of the unit sphere, we define
\begin{equation*}
  \Dcal(F) 
  \,:=\, \bigl\{ \psi \in\LtwoSd \;\big|\; 
  \|\psi\|_\LtwoSd \leq \delta / \omegad^{1/2} \bigr\} \,,
\end{equation*}
where $\delta > 0$ is as in
Proposition~\ref{pro:Wellposedness}. 
Then any $f=Hg$ with $g\in\Dcal(F)$ satisfies $f\in U_{\delta}$, and 
the unique small radiating solution $v = V(Hg)$ of
\eqref{eq:GeneralSP} has the far field pattern 
$\vinfty = (V(Hg))^\infty$.
Introducing the \emph{nonlinear far field operator} 
$F:\Dcal(F)\tm\LtwoSd \to \LtwoSd$ by 
\begin{equation}
  \label{eq:FarfieldOperator}
  F(g) 
  \,:=\, (V(Hg))^\infty \,,
\end{equation}
we obtain from \eqref{eq:Farfield} that
\begin{equation*}
  F(g)
  \,=\, H^* \bigl( 
  k^2 q(\ph,|v + Hg|)(v + Hg) \bigr) \,.
\end{equation*}
These facts are summarized as follows.

\begin{proposition}
  \label{pro:NonlinearFactorization}
  Suppose that Assumption~\ref{ass:Index1} holds, and let
  $g\in\Dcal(F)$.  
  Then the far field pattern of the unique small radiating solution
  $V(f)$ of \eqref{eq:GeneralSP} with $f=Hg$ satisfies 
  \begin{equation}
    \label{eq:Factorization}
    F(g)
    \,=\, H^* T(Hg) \,,
  \end{equation}
  where $T: \Dcal(T) \tm \LtwoD \to \LtwoD$ is defined by 
  \begin{equation}
    \label{eq:OperatorT}
    T(f)(x)
    \,=\, k^2q\bigl(x,|V(f)(x)+f(x)|\bigr)(V(f)(x)+f(x)) \,, \qquad
    x\in{D} \,.
  \end{equation}
  Here $\Dcal(T) := \ol{H(\Dcal(F))}$.
\end{proposition}

\begin{remark}
  In the linear case when $q=q_0$, the far field operator 
  $F_0:\LtwoSd\to\LtwoSd$ is given by
  \begin{equation*}
    F_0g
    \,:=\,  (V_0Hg)^\infty \,.
  \end{equation*}
  The factorization \eqref{eq:Factorization} reads
  \begin{equation*}
    F_0g 
    \,=\, H^*T_0Hg \,, \qquad g\in\LtwoSd \,,
  \end{equation*}
  where $T_0: \LtwoD \to \LtwoD$ is defined by 
  \begin{equation}
    \label{eq:DefT0}
    T_0f 
    \,:=\, k^2 q_0(f+V_0f) \,.
  \end{equation}
  Then \eqref{eq:Assumption3a} implies that, for any $f\in\Dcal(T)$,
  \begin{equation*}
    \begin{split}
      &\|T(f) - T_0f\|_\LtwoD
      \,=\, k^2 \bigl\|q(\ph,|V(f)+f|)(V(f)+f) - q_0 (V_0f+f)\bigr\|_\LtwoD \\
      &\,\leq\, k^2 \bigl\|q(\ph,|V(f)+f|)(V(f)+f)-q_0 (V(f)+f)\bigr\|_\LtwoD
      + k^2 \|q_0(V(f)-V_0f)\|_\LtwoD \\
      &\,\leq\, k^2C_q \bigl\| |V(f)+f|^{1+\alpha} \bigr\|_\LtwoD 
      + k^2 \|q_0\|_\LinftyD \| V(f)-V_0f \|_\LtwoD \,.   
    \end{split}
  \end{equation*}
  Applying \eqref{eq:VHoelderLtwo} and
  \eqref{eq:VStabilityLinfty}--\eqref{eq:VStabilityLtwo} gives
  \begin{equation}
    \label{eq:BoundTMinusT0Ltwo}
    \begin{split}
      \|T(f) -T_0f\|_\LtwoD
      &\,\leq\, k^2C_q\|V(f)+f\|_\LinftyD^\alpha \|V(f)+f\|_\LtwoD 
      + C \|f\|_\LinftyD^\alpha \|f\|_\LtwoD \\
      &\,\leq\, C \|f\|_\LinftyD^\alpha \|f\|_\LtwoD \,.   
    \end{split}
  \end{equation}
  Similarly, using \eqref{eq:VHoelderLinfty} and
  \eqref{eq:VStabilityLinfty}, we find that, for any 
  $f\in\Dcal(T)$, 
  \begin{equation}
    \label{eq:BoundTMinusT0Linfty}
    \|T(f) -T_0f\|_\LinftyD
    \,\leq\, C \|f\|_\LinftyD^{\alpha+1} \,.
  \end{equation}
  \hfill$\lozenge$
\end{remark}

\section{Uniqueness for the inverse scattering problem}
\label{sec:InverseProblem}
In this section we restrict the discussion to generalized Kerr-type
nonlinearities $q$ as in~\eqref{eq:generalizedKerr}.
We show that the knowledge of the nonlinear far field operator
uniquely determines the associated nonlinear refractive index.
A related result has recently been established for a different class
of real analytic nonlinearities in \cite{Fur20}.

\begin{theorem}
  \label{thm:UniquenessInverseProblem}
  For $j=1,2$ let
  \begin{equation}
    \label{eq:MaterialLawUniqueness}
    q^{(j)}(x,|z|) 
    \,=\, q_0^{(j)}(x) + \sum_{l=1}^L q_l^{(j)}(x)|z|^{\alpha_l} 
    \qquad x\in\Rd \,,\; z\in\C \,,\; j=1,2 \,,
  \end{equation}
  be a generalized Kerr-type nonlinearity, where
  $q_0^{(j)},\ldots,q_L^{(j)}\in \LinftyRd$ with support in $D$,
  the lowest order term satisfies $\essinf q_0^{(j)}>-1$, 
  and the exponents fulfill $0<\alpha_1<\cdots<\alpha_L<\infty$. 
  If the associated nonlinear far field operators satisfy
  $F^{(1)}=F^{(2)}$, then~$q^{(1)} = q^{(2)}$.
\end{theorem}

\begin{proof}
  By linearization around zero we first show that
  $q_0^{(1)}=q_0^{(2)}$ in \eqref{eq:MaterialLawUniqueness}. 
  We consider factorizations of the far field operators 
  $F^{(j)}=H^* T^{(j)}(H)$, $j=1,2$, as in 
  Proposition~\ref{pro:NonlinearFactorization}, where $H$ and~$H^*$
  are the Herglotz operator and its adjoint from
  \eqref{eq:HerglotzOperator} and \eqref{eq:HerglotzAdjoint}, and the 
  operator $T^{(j)}$ is as in \eqref{eq:OperatorT} with $q$ replaced
  by $q^{(j)}$.
  Furthermore, we denote by $T_0^{(j)}$ the bounded linear operator
  from \eqref{eq:DefT0} with $q_0$ replaced by $q_0^{(j)}$.
  Then \eqref{eq:BoundTMinusT0Linfty} shows that 
  $T^{(j)}(f) = T_0^{(j)} f + O(\|f\|_\LinftyD^{\alpha_1+1})$
  as~$\|f\|_\LinftyD\to 0$. 
  Recalling \eqref{eq:BoundHerglotzLinfty}, we obtain from
  $F^{(1)}=F^{(2)}$ that 
  \begin{equation*}
    F_0^{(1)}
    \,=\, H^* T_0^{(1)} H 
    \,=\, H^* T_0^{(2)} H 
    \,=\, F_0^{(2)} \,,
  \end{equation*}
  where $F_0^{(j)}$ is the linear far field operator corresponding to
  the contrast function $q_0^{(j)}$, $j=1,2$.
  The uniqueness of solutions to the inverse medium scattering problem
  for the linear Helmholtz equation  (see, e.g.,
  \cite[Thm.~7.28]{Kir21} or \cite{Buk08,Nac88,Nov88,Ram88}) implies
  that $q_0^{(1)}=q_0^{(2)}=:q_0$.
  In particular we conclude that $T_0^{(1)}=T_0^{(2)}=:T_0$. 

  To prove the theorem by induction, we now assume 
  $q_l^{(1)}=q_l^{(2)}=:q_l$ for $l=0,\ldots,m-1$, where
  $m\in\{1,\ldots,L\}$. 
  The nonlinear Lippmann-Schwinger equation \eqref{eq:GeneralLS}
  gives, for $f\in U_{\delta}$ and $j=1,2$,  
  \begin{equation*}
    V^{(j)}(f) 
    \,=\, k^2\Phi_k\ast q^{(j)}\bigl(\ph,|V^{(j)}(f)+f|\bigr)(V^{(j)}(f)+f)  
    \qquad \text{in } {D} \,.
  \end{equation*}
  Here,  $V^{(j)}(f)$ stands for the solution map $V(f)$ from
  Proposition~\ref{pro:Wellposedness} with $q$ replaced by $q^{(j)}$.
  Setting $\alpha_0:= 0$ and using \eqref{eq:MaterialLawUniqueness} we
  obtain that 
  \begin{equation}
    \label{eq:ProofUniqueness7}
    \begin{split}
      &V^{(1)}(f)-V^{(2)}(f) \\ 
      &\,=\, k^2\Phi_k\ast \bigl( 
      q^{(1)}(\ph,|V^{(1)}(f)+f|)(V^{(1)}(f)+f)
      - q^{(2)}(\ph,|V^{(2)}(f)+f|)(V^{(2)}(f)+f) \bigr) \\
      &\,=\, k^2\Phi_k\ast \biggl( \sum_{l=0}^{m-1} q_l
      \Bigl( |V^{(1)}(f)+f|^{\alpha_l}(V^{(1)}(f)+f) 
      - |V^{(2)}(f)+f|^{\alpha_l}(V^{(2)}(f)+f) \Bigr) \biggr) \\
      &\phantom{\,=\,}
      + k^2\Phi_k\ast \biggl( \sum_{l=m}^{L} \Bigl(
      q_l^{(1)} |V^{(1)}(f)+f|^{\alpha_l}(V^{(1)}(f)+f) 
      - q_l^{(2)}|V^{(2)}(f)+f|^{\alpha_l}(V^{(2)}(f)+f) 
      \Bigr) \biggr)
    \end{split}
  \end{equation}
  in ${D}$.
  Applying Lemma~\ref{lmm:UsefulEstimate} and
  \eqref{eq:VStabilityLinfty} we find that, for $l=1,\ldots,m-1$,
  \begin{equation}
    \label{eq:ProofUniqueness8}
    \begin{split}
      \bigl| |V^{(1)}(f)+f|^{\alpha_l}(V^{(1)}(f)+f) 
      -& |V^{(2)}(f)+f|^{\alpha_l}(V^{(2)}(f)+f) \bigr|\\
      &\,\leq\, C \bigl( |f| + |V^{(1)}(f)| + |V^{(2)}(f)| \bigr)^{\alpha_l}
      \bigl| V^{(1)}(f) - V^{(2)}(f) \bigr| \\
      &\,\leq\, C \|f\|_\LinftyD^{\alpha_l}
      \bigl| V^{(1)}(f) - V^{(2)}(f) \bigr| \,.
    \end{split}
  \end{equation}
  Accordingly, 
  \begin{multline*}
    \sum_{l=0}^{m-1} q_l \Bigl( |V^{(1)}(f)+f|^{\alpha_l}(V^{(1)}(f)+f) 
    - |V^{(2)}(f)+f|^{\alpha_l}(V^{(2)}(f)+f) \Bigr) \\
    \,=\, \qtilde_{f,m-1} \bigl(V^{(1)}(f) - V^{(2)}(f)\bigr) \,,
  \end{multline*}
  where $\qtilde_{f,m-1}\in L^\infty(\Rd)$ is given by
  \begin{equation*}
    \qtilde_{f,m-1}
    := q_0
    + \sum_{l=1}^{m-1} \! q_l \frac{
      |V^{(1)}(f)+f|^{\alpha_l}(V^{(1)}(f)+f) 
      - |V^{(2)}(f)+f|^{\alpha_l}(V^{(2)}(f)+f) }
    {V^{(1)}(f) - V^{(2)}(f)} {\bf 1}_{V^{(1)}(f) \neq V^{(2)}(f)} \,.
  \end{equation*}
  We note that $\qtilde_{f,m-1}$ is supported in $\ol{D}$ and
  \eqref{eq:ProofUniqueness8} implies that
  \begin{equation}
    \label{eq:ProofUniqueness9}
    \|\qtilde_{f,m-1}-q_0\|_\LinftyD 
    \,\leq \, C\|f\|_\LinftyD^{\alpha_1} \,.
  \end{equation}
  Hence, for $f\in U_{\delta}$ such that $\|f\|_\LinftyD$ is
  sufficiently small, we conclude from \eqref{eq:EstLSLinfty} that the
  operator $I-k^2\Phi_k\ast (\qtilde_{f,m-1}\ph):\LinftyD\to \LinftyD$
  is invertible with a uniform bound for the operator norm of the
  inverse (see, e.g., \cite[Thm.~10.1]{Kre14}). 
  Denoting 
  \begin{equation}
    \label{eq:DefR(f)}
    R(f) 
    \,:=\, \sum_{l=m}^L
    \Bigl( q_l^{(1)} |V^{(1)}(f)+f|^{\alpha_l}(V^{(1)}(f)+f) 
    - q_l^{(2)}|V^{(2)}(f)+f|^{\alpha_l}(V^{(2)}(f)+f) \Bigr)
  \end{equation}
  we find from \eqref{eq:ProofUniqueness7}, Young's inequality, and
  \eqref{eq:VStabilityLinfty} that 
  \begin{equation}
    \label{eq:ProofUniqueness11}
    \begin{split}
      \bigl\|V^{(1)}(f)-V^{(2)}(f)\bigr\|_\LinftyD
      &\,=\, \bigl\| \bigl( I-k^2\Phi_k\ast (\qtilde_{f,m-1}\ph) \bigr)^{-1}
      (k^2\Phi_k\ast  R(f)) \bigr\|_\LinftyD \\
      &\,\leq\, C \|k^2\Phi_k\ast  R(f)\|_\LinftyD \\
      &\,\leq\, C \|\Phi_k\|_{L^1(B_{2R}(0))} \|R(f)\|_\LinftyD \\
      &\,\leq\, C \bigl(
      \| V^{(1)}(f) + f \|_\LinftyD^{\alpha_m+1}
      + \| V^{(2)}(f) + f \|_\LinftyD^{\alpha_m+1} \bigr) \\
      &\,\leq\, C \|f\|_\LinftyD^{\alpha_m+1} \,.
    \end{split}
  \end{equation}
  Here, $R>0$ was chosen such that $D\tm\BR$.
  
  Next we want to use \eqref{eq:ProofUniqueness11} in order to deduce 
  $q_m^1= q_m^2$. Set $w_f:= V^{(1)}(f)-V^{(2)}(f)$. 
  By assumption we know that the far field of $w_f$ vanishes
  whenever~$f$ is a sufficiently small Herglotz wave, i.e., when 
  $f\in U_{\delta}\cap\Rcal(H)$ with $\delta>0$ as in
  Proposition~\ref{pro:Wellposedness}. 
  Moreover, we find as in the proof of
  Lemma~\ref{lmm:LippmannSchwinger} that~\eqref{eq:ProofUniqueness7} 
  implies
  \begin{equation*}
    \Delta w_f + k^2(1+\qtilde_{f,m-1})w_f 
    \,=\, -k^2 R(f) \qquad \text{in } \Rd \,,
  \end{equation*}
  in particular $\Delta w_f+k^2w_f= 0$ in $\Rd\setminus\ol{\BR}$ and
  $w_f$ is radiating. 
  So Rellich's lemma (see, e.g., \cite[Lmm.~2.12]{ColKre19}) gives 
  $w_f = 0$ in $\Rd\setminus\ol{\BR}$.
  Now let $v\in H^2(\BR)$ be any solution of
  $\Delta v+k^2(1+q_0)v= 0$ in $\BR$. 
  Then, for all $f\in U_{\delta}\cap\Rcal(H)$,
  \begin{equation*}
    \begin{split}
      0
      &\,=\, \int_{\di\BR} \Bigl( w_f \frac{\di v}{\di\nu} 
      - v\frac{\di w_f}{\di\nu} \Bigr) \ds \\
      &\,=\, \int_{\BR} \bigl( w_f \Delta v - v\Delta w_f \bigr) \dx \\ 
      &\,=\, \int_{\BR} \Bigl( w_f \bigl(-k^2(1+q_0)v\bigr) 
      - v\bigl(-k^2(1+\qtilde_{f,m-1})w_f-k^2R(f)\bigr) \Bigr) \dx \\
      &\,=\, \int_{\BR} v 
      \bigl( k^2R(f) + k^2(\qtilde_{f,m-1}-q_0)w_f \bigr) \dx  \\
      &\,=\, \int_{D} v 
      \bigl( k^2R(f) + k^2(\qtilde_{f,m-1}-q_0)w_f \bigr) \dx \,. 
    \end{split}
  \end{equation*}
  In the last equality we used that $R(f)$ and $\qtilde_{f,m-1}-q_0$
  are supported in $\ol{D}$ by our assumption on the nonlinear
  contrast function. 
  In \eqref{eq:ProofUniqueness11} we found that
  $\|w_f\|_\LinftyD \leq C \|f\|_\LinftyD^{\alpha_m+1}$,  
  and combining this with \eqref{eq:ProofUniqueness9} gives
  \begin{equation}
    \label{eq:OrthogonalityR(f)a}
    0
    \,=\, \int_{D} v \, R(f) \dx  
    + O\bigl(\|f\|_\LinftyD^{\alpha_m+\alpha_1+1}\bigr) 
    \qquad \text{as } \|f\|_\LinftyD \to 0 \,.
  \end{equation}
  Next we identify the leading order term in $R(f)$. 
  Using Lemma~\ref{lmm:UsefulEstimate} and \eqref{eq:VStabilityLinfty}, 
  \eqref{eq:VHoelderLinfty} we obtain that, for~$j=1,2$, 
  \begin{equation}
    \label{eq:OrthogonalityR(f)b}
    \begin{split}
      &\bigl| |V^{(j)}(f)+f|^{\alpha_m}(V^{(j)}(f)+f) 
      -|V_0f+f|^{\alpha_m}(V_0f+f) \bigr|\\
      &\,\leq\, C \bigl( |f| + |V^{(j)}(f)| + |V_0f| \bigr)^{\alpha_m}
      \bigl| V^{(j)}(f) - V_0^{(j)}f \bigr| \\
      &\,\leq\, C \|f\|_\LinftyD^{\alpha_m+\alpha_1+1} \,. 
    \end{split}
  \end{equation}
  Similarly, we find that, for $j=1,2$ and $l=m+1,\ldots,L$,
  \begin{equation}
    \label{eq:OrthogonalityR(f)c}
    \bigl| |V^{(j)}(f)+f|^{\alpha_l}(V^{(j)}(f)+f) \bigr|
    \,\leq\, C \|f\|_\LinftyD^{\alpha_{m+1}+1} \,.
  \end{equation}
  Substituting \eqref{eq:DefR(f)} into \eqref{eq:OrthogonalityR(f)a},
  and applying
  \eqref{eq:OrthogonalityR(f)b}--\eqref{eq:OrthogonalityR(f)c} gives 
  \begin{equation*}
    0
    \,=\, \int_{D} v \, 
    \bigl( q_m^{(1)} - q_m^{(2)} \bigr)
    |V_0f+f|^{\alpha_m}(V_0f+f) 
    \dx  
    + O\bigl(\|f\|_\LinftyD^{\alpha_m+\alpha_1+1}\bigr)
    + O\bigl(\|f\|_\LinftyD^{\alpha_{m+1}+1}\bigr) \,.
  \end{equation*}
  as $\|f\|_\LinftyD\to 0$. 
  Hence, for all $f\in U_{\delta}\cap\Rcal(H)$,
  \begin{equation*}
    0
    \,=\, \int_{D} v \, \bigl(q_m^{(1)}-q_m^{(2)}\bigr)
    |V_0f+f|^{\alpha_m} (V_0f+f) \dx \,.
  \end{equation*}
  Setting $f=f_1+tf_2$ for $f_1,f_2\in U_{\delta}\cap\Rcal(H)$ and
  differentiating with respect to $t$ gives 
  \begin{equation*}
    \begin{split}
      0 
      \,=\, \int_{D} v \bigl(q_m^{(1)}-q_m^{(2)}\bigr)
      Z(f_1,f_2) \dx \,,
    \end{split}
  \end{equation*}
  where
  \begin{equation*}
    \begin{split}
      Z(f_1,f_2)
      &\,:=\, \Bigl(1+\frac{\alpha_m}{2}\Bigr)|V_0f_1+f_1|^{\alpha_m} 
      (V_0f_2+f_2)\\
      &\phantom{\,:=\,}
      + \frac{\alpha_m}{2} |V_0f_1+f_1|^{\alpha_m-2}(V_0f_1+f_1)^2\,
      \ol{(V_0f_2+f_2)} \,.
    \end{split}
  \end{equation*}
  Since $\rmi f_1 \in U_{\delta}\cap\Rcal(H)$, too, we even get
  \begin{equation}
    \label{eq:Orthogonalityq1minusq2}
    \begin{split}
      0 
      &\,=\, \int_{D} v \, \bigl(q_m^{(1)}-q_m^{(2)}\bigr) 
      \bigl( Z(f_1,f_2) + Z(\rmi f_1,f_2) \bigr) \dx \\
      &\,=\, (2+\alpha_m) \int_{D} v \, 
      \bigl(q_m^{(1)}-q_m^{(2)}\bigr) |V_0f_1+f_1|^{\alpha_m} 
      (V_0f_2+f_2) \dx \,.
    \end{split}
  \end{equation}
  Next we recall that the span of all total fields $f+V_0f$ that
  correspond to radiating solutions~$V_0f$ of the linear scattering
  problem \eqref{eq:LinearSP} with Herglotz incident fields $f=Hg$,
  $g\in\LtwoSd$, is dense in the space of solutions to the linear
  Helmholtz equation in
  \begin{equation*}
    \Delta \vtilde + k^2(1+q_0)\vtilde 
    \,=\, 0 \qquad \text{in } \BR 
  \end{equation*}
  with respect to the $\LtwoBR$-norm where $D\subset \BR$ (see
  \cite[Thm.~7.24]{Kir21}, where this result has been shown for plane
  wave incident fields instead of Herglotz incident fields).
  Since $f_1,f_2\in U_{\delta}\cap\Rcal(H)$ have been arbitrary in
  \eqref{eq:Orthogonalityq1minusq2}, we get for all solutions
  $v,\vtilde\in H^2(\BR)$ of $\Delta v + k^2(1+q_0)v = 0$ in $\BR$
  that 
  \begin{equation*}
    0 
    \,=\,  \int_{D} v \vtilde 
    \bigl(q_m^{(1)}-q_m^{(2)}\bigr) |V_0f_1+f_1|^{\alpha_m} \dx \,.
  \end{equation*}
  This gives $(q_m^{(1)}-q_m^{(2)})|V_0f_1+f_1|^{\alpha_m}=0$ for any
  $f_1 \in U_{\delta}\cap\Rcal(H)$ (see, e.g.,
  \cite[Thm.~7.27]{Kir21} or \cite{Buk08,Nac88,Nov88,Ram88}).
  From this we infer $(q_m^{(1)}-q_m^{(2)})(V_0f_1+f_1)=0$ for any
  $f_1 \in U_{\delta}\cap\Rcal(H)$ and thus
  \begin{equation*}
    \int_D (q_m^{(1)}-q_m^{(2)})(V_0f_1+f_1)(V_0f_2+f_2) \dx
    \,=\, 0 
  \end{equation*}
  for any given $f_1,f_2 \in U_{\delta}\cap\Rcal(H)$.
  The density result used above shows   
  $q_m^{(1)}-q_m^{(2)}=0$ a.e.\ in~$D$, and thus
  $q_m^{(1)}=q_m^{(2)}$. 
  So the claim is proven by induction.
\end{proof}

\section{The nonlinear factorization method}
\label{sec:FactorizationMethod}
In this section we discuss a generalization of the factorization
method to recover the shape of a nonlinear scattering object from
observations of the corresponding nonlinear far field operator. 
We consider general nonlinear contrast functions 
$q\in L^\infty(\Rd\times\R)$ as in Section~\ref{sec:Setting}, but here
we make the following slightly stronger assumptions. 

\begin{assumption}
  \label{ass:Index2}
  Let $D$ be open and Lipschitz bounded such that
  $\Rd\setminus\ol{D}$ is connected. 
  Then the nonlinear contrast function $q\in L^\infty(\Rd\times\R)$
  shall satisfy Assumption~\ref{ass:Index1}, and
  \begin{itemize}
  \item[(i)] $\supp(q) \tm \ol{D}\times\R$,
  \item[(ii)] $\supp(q_0) = \ol{D}$ with
    $q_0 \geq \qmin > 0$ a.e.\ in $D$ for some $\qmin>0$,
  \item[(iii)] the wave number $k^2$ is such that the homogeneous
    linear transmission eigenvalue problem to determine 
    $v,w\in \LtwoD$, $(v,w)\not=(0,0)$ with
    \begin{align*}
      \Delta v + k^2 v 
      &\,=\, 0 \quad \text{in } D \,, 
      & v 
      &\,=\, w \quad \text{on } \di D \,,\\
      \Delta w + k^2 (1+q_0) w 
      &\,=\, 0 \quad \text{in } D \,, 
      & \frac{\di v}{\di\nu} 
      &\,=\, \frac{\di w}{\di\nu} \quad \text{on } \di D \,,
    \end{align*}
    (see, e.g., \cite[Def.~7.21]{Kir21}) has no
    nontrivial solution.
  \end{itemize}
\end{assumption}

A factorization method for nonlinear weakly scattering objects and for
scattering objects with small nonlinearity of linear growth has
already been discussed in \cite{Lec11}.
In constrast to this work, we consider a larger class of nonlinear
refractive indices without any smallness assumption on the a priori
unknown  nonlinearity, but we assume that the incident fields that are
used for the reconstruction are small relative to the size of the
nonlinearity. 

Let $\delta>0$ be as in Proposition~\ref{pro:Wellposedness}. 
We consider the nonlinear far field operator~$F$
from~\eqref{eq:FarfieldOperator} with the factorization $F=H^*T(H)$
from Proposition~\ref{pro:NonlinearFactorization}.
The next theorem is a nonlinear version of the abstract inf-criterion 
of the factorization method to describe the range of $H^*$ in terms
of~$F$. 
This result has been established in \cite[Thm.~2.1]{Lec11}. 
The proof is essentially the same as in the linear case (see, e.g.,
\cite[Lmm.~7.33]{Kir21}).

\begin{theorem}
  \label{thm:AbstractInfCriterion}
  Let $X$ and $Y$ be Hilbert spaces, $\rho>0$, and let 
  \begin{equation*}
    \Fcal:\Dcal(\Fcal) 
    \,:=\, \{ g\in X \;|\; \|g\|_X \leq \rho \} 
    \tm X \to X  
  \end{equation*}
  be a nonlinear operator.
  We assume that $\Fcal = \Hcal^*\Tcal(\Hcal)$, where $\Hcal:Y\to X$
  is a compact linear operator and $\Tcal: \Dcal(\Tcal)\tm Y\to Y$
  with $\Dcal(\Tcal) = \ol{\Hcal(\Dcal(\Fcal))}$ satisfies
  \begin{equation*}
    \|\Tcal(\Hcal g)\|_Y 
    \,\leq\, C_* \|\Hcal g\|_Y 
  \end{equation*}
  and
  \begin{equation*}
    |\skp{\Tcal(\Hcal g)}{\Hcal g}_Y| 
    \,\geq\, c_* \|\Hcal g\|_Y^2
  \end{equation*}
  for all $g\in\Dcal(\Fcal)$ with $\|g\|_X\leq\rho$ and some
  $c_*,C_*>0$. 
  Then, for any $\phi\in X$, $\phi\not=0$, and
  any~$0<\rhotilde\leq\rho$,
  \begin{equation}
    \label{eq:InfCharacterization}
    \phi \in\Rcal(\Hcal^*) 
    \quad\Llra\quad
    \inf \biggl\{ 
    \Bigl| \frac{\skp{\Fcal(g)}{g}_X}{\skp{g}{\phi}_X^2} \Bigr|
    \;\bigg|\; g\in \Dcal(\Fcal)\tm X \,,\; \|g\|_X=\rhotilde \,,\; 
    \skp{g}{\phi}_X \not=0 \biggr\}
    \,>\, 0 \,.
  \end{equation}
\end{theorem}

\begin{proof}
  Let $0\not=\phi=\Hcal^*\psi\in\Rcal(\Hcal^*)$ for some $\psi\in Y$.
  Then $\psi\not=0$, and for any $g\in\Dcal(\Fcal)\tm X$ with 
  $\|g\|_X=\rhotilde\leq\rho$ and $\skp{g}{\phi}_X\not=0$ we
  find that 
  \begin{equation*}
    \begin{split}
      |\skp{\Fcal(g)}{g}_X|
      &\,=\, |\skp{\Hcal^*\Tcal(\Hcal g)}{g}_X|
      \,=\, |\skp{\Tcal(\Hcal g)}{\Hcal g}_Y|
      \,\geq\, c_* \|\Hcal g\|_Y^2\\
      &\,=\, \frac{c_*}{\|\psi\|_Y^2} \|\Hcal g\|_Y^2 \|\psi\|_Y^2
      \,\geq\, \frac{c_*}{\|\psi\|_Y^2} |\skp{\Hcal g}{\psi}_Y|^2\\
      &\,=\, \frac{c_*}{\|\psi\|_Y^2} |\skp{g}{\Hcal^*\psi}_X|^2
      \,=\, \frac{c_*}{\|\psi\|_Y^2} |\skp{g}{\phi}_X|^2
    \end{split}
  \end{equation*}
  Thus we have found a positive lower bound for the infimum in
  \eqref{eq:InfCharacterization}. 
  
  Now let $0\not=\phi\not\in\Rcal(\Hcal^*)$.
  We first show that the subspace
  $\{ \Hcal g \;|\; g\in X \,,\; \skp{g}{\phi}_X=0 \}$ is
  dense in $\Rcal(\Hcal)$. 
  Let $\psi\in\Rcal(\Hcal)$ such that 
  $0=\skp{\Hcal g}{\psi}_Y=\skp{g}{\Hcal^*\psi}_X$ for all $g\in X$
  with $\skp{g}{\phi}_X=0$. 
  That means $\Hcal^*\psi\in\spann\{\phi\}$, and because
  $\phi\not\in\Rcal(\Hcal^*)$, we conclude that $\Hcal^*\psi=0$. 
  Therefore $\psi \in \Rcal(\Hcal)\cap\Ncal(\Hcal^*)$, i.e., $\psi=0$, and we
  have shown that 
  $\{ \Hcal g \;|\; g\in X \,,\; \skp{g}{\phi}_X=0 \}$ is dense in
  $\Rcal(\Hcal)$. 
  Since $\Hcal\phi/\|\phi\|_Y^2\in\Rcal(\Hcal)$, we can find a sequence
  $(\gtilde_n)_n \tm \{ g \in X \;|\; \skp{g}{\phi}_X=0 \}$ such 
  that~$\Hcal\gtilde_n\to-\Hcal\phi/\|\phi\|_X^2$. 
  Setting $\ghat_n := \gtilde_n+\phi/\|\phi\|_X^2$ this yields 
  $\skp{\ghat_n}{\phi}_X=1$ and $\Hcal\ghat_n\to 0$ as~$n\to\infty$.
  Thus, we define 
  $g_n := \rhotilde\, \ghat_n/\|\ghat_n\|_X \in \Dcal(\Fcal)$
  to obtain 
  \begin{equation*}
    \Bigl| \frac{\skp{\Fcal(g_n)}{g_n}_X}{\skp{g_n}{\phi}_X^2} \Bigr|
    \,=\, \frac{|\skp{\Tcal(\Hcal g_n)}{\Hcal g_n}_Y|}
    {|\skp{g_n}{\phi}_X|^2}
    \,\leq\, \frac{C_* \|\Hcal g_n\|_Y^2}{|\skp{g_n}{\phi}_X|^2}
    \,=\, \frac{C_* \|\Hcal\ghat_n\|_Y^2}{|\skp{\ghat_n}{\phi}_X|^2}
    \to 0 \qquad \text{as } n\to\infty \,,
  \end{equation*}
  i.e., the infimum in \eqref{eq:InfCharacterization} is zero. 
\end{proof}

Next we show that the operator $T$ from
Proposition~\ref{pro:NonlinearFactorization} satisfies the assumptions
in Theorem~\ref{thm:AbstractInfCriterion}. 

\begin{proposition}
  \label{pro:BoundsT}
  Suppose that Assumption~\ref{ass:Index2} holds, and let $\delta>0$
  be as in Proposition~\ref{pro:Wellposedness}. 
  Then there are constants $c_*,C_*,C>0$ such that 
  \begin{subequations}
    \label{eq:BoundsT}
    \begin{align}
      \|T(f)\|_\LtwoD 
      &\,\leq\, C_* \bigl( 1 + \|f\|_\LinftyD^\alpha \bigr) 
        \|f\|_\LtwoD \,,  \label{eq:BoundT} \\
      |\skp{T(f)}{f}_\LtwoD| 
      &\,\geq\, c_* \bigl( 1 - C \|f\|_\LinftyD^\alpha \bigr) 
        \|f\|_\LtwoD^2  \label{eq:CoercivityT}
    \end{align}
  \end{subequations}
  for all $f\in U_{\delta}$. 
\end{proposition}

\begin{proof}
  Let $f\in U_{\delta}$. 
  We first note that \eqref{eq:DefT0} and \eqref{eq:V0EstimateLtwo}
  show that
  \begin{equation*}
    \begin{split}
      \|T_0f\|_\LtwoD 
      \,\leq\, k^2 \|q_0\|_\LinftyD (1+C_{V_0,2}) \|f\|_\LtwoD \,.
    \end{split}
  \end{equation*}
  Combining this with \eqref{eq:BoundTMinusT0Ltwo} gives
  \begin{equation*}
    \begin{split}
      \|T(f)\|_\LtwoD
      &\,\leq\, \|T_0 f\|_\LtwoD + \|T(f)-T_0 f\|_\LtwoD \\
      &\,\leq\, k^2 \|q_0\|_\LinftyD (1+C_{V_0,2}) \|f\|_\LtwoD
      + C \|f\|_\LinftyD^\alpha \|f\|_\LtwoD\\
      &\,\leq\, C_* \bigl( 1 + \|f\|_\LinftyD^\alpha \bigr) 
        \|f\|_\LtwoD
    \end{split}
  \end{equation*}
  for some $C_*>0$.

  Next let $S_0:\LtwoD\to \LtwoD$ be defined by
  \begin{equation*}
    S_0\psi
    \,:=\, \frac{1}{k^2 q_0}\psi - \Phi_k\ast \psi \,.
  \end{equation*}
  It has been shown in \cite[Thm.~7.32]{Kir21} that $S_0$ is an
  isomorphism with  $T_0 =S_0 ^{-1}$, which can be seen
  using \eqref{eq:DefT0} and \eqref{eq:DefV0} as follows.
  Let $h\in\LtwoD$, then
  \begin{equation*}
    \begin{split}
      S_0 T_0 h
      &\,=\, \frac{1}{k^2q_0} T_0 h - \Phi_k\ast (T_0 h)
      \,=\, (I+V_0)h - \Phi_k\ast \bigl(k^2q_0(h+V_0h)\bigr) \\
      &\,=\, h + \bigl(I - k^2 \Phi_k\ast (q_0\ph)\bigr)(V_0h)
      - k^2 \Phi_k\ast (q_0h) 
      \,=\, h \,.   
    \end{split}
  \end{equation*}
  If $k^2$ is not an interior transmission eigenvalue then it follows
  from \cite[Lmm.~7.35]{Kir21} and the arguments used in the proof of
  \cite[Thm.~7.30]{Kir21} that there exists a constant $c_*>0$ such that 
  \begin{equation*}
    | \skp{T_0f}{f}_\LtwoD |
    \,=\, | \skp{S_0^{-1}f}{f}_\LtwoD |
    \,\geq\, c_* \|f\|_\LtwoD^2
    \qquad \text{for all } f\in\ol{\Rcal(H)} \,.
  \end{equation*}
  Accordingly, combining this with \eqref{eq:BoundTMinusT0Ltwo} gives
  \begin{equation*}
    \begin{split}
      \bigl|\skp{T(f)}{f}_\LtwoD\bigr|
      &\,\geq\, \bigl|\skp{T_0f}{f}_\LtwoD\bigr| 
      - \bigl|\skp{T(f)-T_0f}{f}_\LtwoD\bigr| \\
      &\,\geq\, \bigl( c_* - C \|f\|_\LinftyD^\alpha \bigr) 
      \|f\|_\LtwoD^2 
      \,\geq\, c_* \bigl( 1 - C \|f\|_\LinftyD^\alpha \bigr) 
      \|f\|_\LtwoD^2 \,.
    \end{split}
  \end{equation*}
\end{proof}

Combining \eqref{eq:BoundsT} with \eqref{eq:HerglotzOperator} and
applying H\"older's inequality gives the following corollary.

\begin{corollary}
  \label{cor:BoundsT}
  Suppose that Assumption~\ref{ass:Index2} holds.
  Then there are constants $c_*,C_*,C>0$ such that 
  \begin{subequations}
    \label{eq:BoundsTH}
    \begin{align}
      \|T(Hg)\|_\LtwoD 
      &\,\leq\, C_* \bigl( 1+C\omegad^{\alpha/2}\|g\|_\LtwoSd^\alpha \bigr)  
        \|Hg\|_\LtwoD  \,, \label{eq:BoundTH}\\
      |\skp{ T(Hg)}{Hg}_\LtwoD|
      &\,\geq\, c_*\bigl( 1-C\omegad^{\alpha/2} 
        \|g\|_{\LtwoSd)}^\alpha \bigr) \|Hg\|_\LtwoD^2 \label{eq:CoercivityTH}
    \end{align}
  \end{subequations}
  for all $g\in\Dcal(F)$.
\end{corollary}

The following result can be shown analogously to
\cite[Thm.~4.6]{KirGri08}. 
\begin{proposition}
  \label{pro:RangeCharacterizationHstar}
  For any $z\in\Rd$ we define the \emph{test function}
  $\phi_z\in\LtwoSd$ by
  \begin{equation*}
    \phi_z(\xhat) 
    \,:=\, e^{-\rmi k z\cdot\xhat} \,, \qquad \xhat\in\Sd \,.
  \end{equation*}
  Then $z\in D$ if and only if $\phi_z\in\Rcal(H^*)$. 
\end{proposition}

Combining the results above, we obtain the main result of this
section. 
\begin{theorem}
  \label{thm:NonlinearFactorization}
  Suppose that Assumption~\ref{ass:Index2} holds, and let $\delta>0$
  be as in Proposition~\ref{pro:Wellposedness}. 
  Let~$C>0$ be the constant in \eqref{eq:CoercivityTH}, and let
  \begin{equation*}
    \rho
    \,:=\, \min \biggl\{
    \frac{\delta}{\omegad^{1/2}} \,, 
    \frac{1}{\omegad^{1/2}} \Bigl(\frac{1}{2C}\Bigr)^{1/\alpha}
    \biggr\}
  \end{equation*}
  Then, for any $0<\rhotilde\leq \rho$ and $z\in\Rd$,
  \begin{equation}
    \label{eq:NLFCharacterization}
    z\in D 
    \;\Llra\; 
    \inf \biggl\{   
    \Bigl|\frac{\skp{F(g)}{g}_\LtwoSd}{\skp{g}{\phi_z}_\LtwoSd^2}\Bigr| 
    \;\bigg|\; g\in \LtwoSd \,,\,
    \|g\|_\LtwoSd=\rhotilde \,,\, 
    \skp{g}{\phi_z}_\LtwoSd\neq 0 \biggr\} 
    > 0 \,.
  \end{equation}
\end{theorem}

\begin{proof}
  By Proposition~\ref{pro:RangeCharacterizationHstar} we know that
  $z\in D$ is equivalent to $\phi_z\in\Rcal(H^*)$, which, by
  Theorem~\ref{thm:AbstractInfCriterion}, is in turn equivalent to
  the condition on the right hand side of~\eqref{eq:NLFCharacterization} 
  provided that the nonlinear far field operator $F$ admits the
  factorization $F=H^*T(H)$ for $T$ as in
  Theorem~\ref{thm:AbstractInfCriterion}.
  This has been shown in Proposition~\ref{thm:NonlinearFactorization}
  and Corollary~\ref{cor:BoundsT}.
  Note that our choice of $\rho$ guarantees the existence of the far
  field operator (see Proposition~\ref{pro:Wellposedness}) as well 
  as the coercivity estimate in
  Proposition~\ref{thm:NonlinearFactorization} (see
  Corollary~\ref{cor:BoundsT}). 
\end{proof}

We will comment on a numerical implementation of this criterion in
Section~\ref{sec:NumericalExamples} below.
For numerical implementations in the linear case we refer, e.g., to
\cite{Kir99,KirGri08}.

\section{The nonlinear monotonicity method}
\label{sec:MonotonicityMethod}
In this section we consider general nonlinear contrast functions 
$q\in L^\infty(\Rd\times\R)$ as in
Section~\ref{sec:FactorizationMethod}, but we waive the assumption on
$k^2$ not being a transmission eigenvalue.

\begin{assumption}
  \label{ass:Index3}
  Let $D$ be open and Lipschitz bounded such that
  $\Rd\setminus\ol{D}$ is connected. 
  Then the nonlinear contrast function $q\in L^\infty(\Rd\times\R)$
  shall satisfy Assumption~\ref{ass:Index1}, and
  \begin{itemize}
  \item[(i)] $\supp(q) \tm \ol{D}\times\R$,
  \item[(ii)] $\supp(q_0) = \ol{D}$ with
    $0 < \qmin \leq q_0 \leq \qmax < \infty$ a.e.\ in $\ol{D}$ for
    some $\qmin,\qmax>0$. 
  \end{itemize}
\end{assumption}

Given any open and bounded subset $B\tm\Rd$, we define the associated 
\emph{probing operator} $P_B:\LtwoSd \to \LtwoSd$ by
\begin{equation*}
  P_Bg \,:=\, k^2 H_B^*H_Bg \,,
\end{equation*}
where $H_B:L^2(\Sd)\to L^2(B)$ and $H_B^*:L^2(B)\to L^2(\Sd)$ are
given as in \eqref{eq:HerglotzOperator} and \eqref{eq:HerglotzAdjoint}
with  $D$ replaced by $B$.
Accordingly, we find that for all $g\in\LtwoSd$,
\begin{equation}
  \label{eq:IdentityTB}
  \skp{P_Bg}{g}_\LtwoSd 
  \,=\, k^2\int_B |H g|^2 \dx 
  \,=\, k^2 \int_B \biggl| 
  \int_\Sd e^{\rmi k \theta\cdot x} g(\theta) \ds(\theta) \biggr|^2 \dx \,.
\end{equation}
The operator~$P_B$ is bounded, compact, and self-adjoint. 

\begin{theorem}
  \label{thm:NonlinearMonotonicity}
  Suppose that Assumption~\ref{ass:Index3} holds, and let $\delta>0$
  be as in Proposition~\ref{pro:Wellposedness}. 
  Let~$B\tm\Rd$ be open and bounded, and let
  \begin{equation*}
    \rho
    \,:=\, \min \biggl\{
    \frac{\delta}{\omegad^{1/2}} \,,
    \frac{1}{\omegad^{1/2}} \Bigl( \
    \frac{k^2 \qmin}{2\,C} \Bigr)^{\frac1\alpha} 
    \biggr\} \,,
  \end{equation*}
  where $C>0$ is the constant from \eqref{eq:BoundTMinusT0Ltwo}
  and $\delta>0$ is as in Proposition~\ref{pro:Wellposedness}.
  For any $0<\rhotilde\leq\rho$ the following characterization
  of $D$ holds. 

  \begin{itemize}
  \item[(a)] If $B\tm D$, then there exists a finite dimensional
    subspace $\Vcal\tm\LtwoSd$ such that, for
    all~$\beta\leq\frac{\qmin}{2}$, 
    \begin{equation*}
      \beta \skp{P_Bg}{g}_\LtwoSd 
      \,\leq\, \real\bigl(\skp{F(g)}{g}_\LtwoSd\bigr) 
      \quad\text{for all } g\in\Vcalperp \text{ with } 
      \|g\|_\LtwoSd = \rhotilde \,.
    \end{equation*}
  \item[(b)] If $B\not\tm D$, then there is no finite dimensional
    subspace $\Vcal\tm\LtwoSd$ and no $\beta>0$ such that
    \begin{equation*}
      \beta \skp{P_Bg}{g}_\LtwoSd 
      \,\leq\, \real\bigl(\skp{F(g)}{g}_\LtwoSd\bigr) 
      \quad\text{for all } g\in\Vcalperp \text{ with } 
      \|g\|_\LtwoSd = \rhotilde \,.
    \end{equation*}
  \end{itemize}
\end{theorem}

\begin{proof}
  We consider the factorization of the far field operator
  $F=H^*T(H)$ as in \eqref{eq:Factorization}. 
  Accordingly,  the linear far field operator corresponding to the
  contrast function $q_0$ satisfies $F_0 = H^* T_0 H$, and we
  obtain from \eqref{eq:BoundTMinusT0Ltwo} that, for all
  $g\in\Dcal(F)$,  
  \begin{equation*}
    \begin{split}
      \real\biggl(\int_\Sd g\, \ol{F(g)}\ds\biggr)
      &\,=\, \real\biggl(\int_\Sd g\, \ol{F_0g}\ds\biggr)
      + \real\biggl(\int_\Sd g\, \ol{(F-F_0)(g)}\ds\biggr) \\
      &\,\geq\, \real\biggl(\int_\Sd g\, \ol{F_0g}\ds\biggr)
      - C \|Hg\|_\LinftyD^\alpha \|Hg\|_\LtwoD^2 \,.
    \end{split}
  \end{equation*}
  Applying \cite[Thm.~3.2]{GriHar18} with $q_1=0$ and $q_2=q$ we
  find that there exists a finite dimensional subspace
  $\Vcal\tm\LtwoSd$ such that, for all $g\in\Dcal(F)\cap\Vcalperp$,
  \begin{equation*}
    \begin{split}
      \real\biggl(\int_\Sd g\, \ol{F(g)}\ds\biggr)
      &\,\geq\, k^2 \int_D  q_0 |Hg|^2 \dx 
      - C \|Hg\|_\LinftyD^\alpha \int_D |Hg|^2 \dx \\
      &\,\geq\, k^2 \Bigl( \qmin - \frac{C\omegad^{\alpha/2}}{k^2} 
      \|g\|_\LtwoSd^\alpha \Bigr)
      \int_D  |Hg|^2 \dx \,.
    \end{split}
  \end{equation*}
  Assuming that $\|g\|_\LtwoSd = \rhotilde$ we obtain that
  \begin{equation*}
    \real\biggl(\int_\Sd g\, \ol{F(g)}\ds\biggr)
    \,\geq\, k^2 \frac{\qmin}{2} \int_D  |Hg|^2 \dx \,.
  \end{equation*}
  Moreover, if $B\tm D$ and $\beta\leq \frac{\qmin}{2}$, then 
  \begin{equation*}
    \beta \int_\Sd g \ol{P_Bg} \ds
    \,=\, k^2 \beta \int_B |Hg|^2 \dx
    \,\leq\, k^2 \frac{\qmin}{2} \int_D  |Hg|^2 \dx \,,
  \end{equation*}
  which shows part (a).

  We prove part (b) by contradiction.  
  Let $B\not\tm D$, $\beta>0$, and assume that
  \begin{equation}
    \label{eq:ProofShape1-3}
    \beta \skp{P_Bg}{g}_\LtwoSd 
    \,\leq\, \real\bigl(\skp{F(g)}{g}_\LtwoSd\bigr) 
    \qquad\text{for all } g\in\Vcal_1^\perp 
    \text{ with } \|g\|_\LtwoSd=\rhotilde 
  \end{equation}
  for some $0<\rhotilde\leq\rho$ and a finite dimensional
  subspace $\Vcal_1\tm\LtwoSd$. 
  Using \eqref{eq:BoundTMinusT0Ltwo} we find that 
  \begin{equation*}
    \begin{split}
      \real\biggl(\int_\Sd g\, \ol{F(g)}\ds\biggr)
      &\,=\, \real\biggl(\int_\Sd g\, \ol{F_0g}\ds\biggr)
      + \real\biggl(\int_\Sd g\, \ol{(F-F_0)(g)}\ds\biggr) \\
      &\,\leq\, \real\biggl(\int_\Sd g\, \ol{F_0g}\ds\biggr)
      + C \|Hg\|_\LinftyD^\alpha \|Hg\|_\LtwoD^2 \,.
    \end{split}
  \end{equation*}
  Applying the monotonicity relation (3.3) in
  \cite[Cor.~3.4]{GriHar18} with $q_1=0$ and $q_2=q$, shows that
  there exists a finite dimensional subspace $\Vcal_2\tm\LtwoSd$ such
  that 
  \begin{equation}
    \label{eq:ProofShape1-5}
    \real\biggl(\int_\Sd g \, \ol{F_0g} \ds\biggr)
    \,\leq\, 
    k^2 \int_D  q_0 |V_0Hg|^2\dx 
    \qquad\text{for all } g\in\Vcal_2^\perp \,.  
  \end{equation}
  Combining \eqref{eq:ProofShape1-3}--\eqref{eq:ProofShape1-5}, we
  obtain for $\Vcaltilde := \Vcal_1^\perp + \Vcal_2^\perp$ that, for
  all $g\in\Vtilde^\perp$ with $\|g\|_\LtwoSd = \rhotilde$, 
  \begin{equation*}
    \begin{split}
      k^2 \beta \|Hg\|_{L^2(B)}^2
      &\,\leq\, k^2 \int_D   q_0 |V_0Hg|^2\dx
      + C \|H g\|_\LinftyD^\alpha \|H g\|_\LtwoD^2\\ 
      &\,\leq\, k^2 \qmax \int_D  |V_0Hg|^2\dx 
      + C \|H g\|_\LinftyD^\alpha \|H g\|_\LtwoD^2 \,.
    \end{split}
  \end{equation*}
  Applying \cite[Thm.~4.5]{GriHar18} with $q_1=0$ and $q_2=q$, this
  implies that there exists a constant~$\Ctilde>0$ such that, for all
  $g\in\Vcaltilde^\perp$ with $\|g\|_\LtwoSd=\rhotilde$, 
  \begin{equation}
    \label{eq:ProofShape1-Contradiction}
    \begin{split}
      k^2 \beta \|Hg\|_{L^2(B)}^2
      &\,\leq\, \bigl( \Ctilde k^2 \qmax 
      + C \|Hg\|_\LinftyD^\alpha \bigr)
      \|Hg\|_\LtwoD^2 \\
      &\,\leq\, \bigl( \Ctilde k^2 \qmax 
      + C \omegad^{\alpha/2} \|g\|_\LtwoD^\alpha \bigr)
      \|Hg\|_\LtwoD^2 \,.
    \end{split}
  \end{equation}
  In the following we denote by $\Pcal_\Vcal:\LtwoSd\to\LtwoSd$ the
  orthogonal projection onto $\Vcal$. 
  Using \cite[Lmm.~4.4]{GriHar18} we obtain as in the proof of
  \cite[Thm.~4.1]{GriHar18} a sequence
  $(\widetilde{g}_m)_{m\in\N}\tm\LtwoSd$ such that 
  $\|\widetilde{g}_m\|_\LtwoSd = \rho/2$, and 
  \begin{equation*}
    \|H\widetilde{g}_m\|_{L^2(B)}
    \,\geq\, m \bigl( \|H\widetilde{g}_m\|_\LtwoD 
    + \|\Pcal_\Vcal\widetilde{g}_m\|_\LtwoSd \bigr) \,, \qquad m\in\N \,.
  \end{equation*}
  Therefore, $g_m:=\widetilde{g}_m-\Pcal_\Vcal\widetilde{g}_m\in\Vcalperp$
  and by rescaling $\widetilde{g}_m$ we can assume without loss of
  generality that $\|g_m\|_\LtwoSd = \rhotilde \leq \rho$.
  Accordingly, if $\|H \|\leq 1$, then
  \begin{equation*}
    \begin{split}
      \|Hg_m\|_{L^2(B)} 
      &\,\geq\, \|H\widetilde{g}_m\|_{L^2(B)} 
      - \|H_B\|\|\Pcal_\Vcal\widetilde{g}_m\|_\LtwoSd\\
      &\,>\, m \|H\widetilde{g}_m\|_\LtwoD 
      + m \|\Pcal_\Vcal\widetilde{g}_m\|_\LtwoSd 
      - \|H_B\|\|\Pcal_\Vcal\widetilde{g}_m\|_\LtwoSd\\
      &\,\geq\, m \|Hg_m\|_\LtwoD 
      + \bigl( m (1 - \|H \|) - \|H_B\| \bigr)
      \|\Pcal_\Vcal\widetilde{g}_m\|_\LtwoSd \\
      &\,\geq\, m \|Hg_m\|_\LtwoD      
    \end{split}
  \end{equation*}
  for all $m\in\N$ such that $m \geq \|H_B\| / (1-\|H \|)$.
  On the other hand, if $\|H \|>1$, then 
  \begin{equation*}
    \begin{split}
      \|Hg_m\|_{L^2(B)} 
      &\,\geq\, \|H\widetilde{g}_m\|_{L^2(B)} 
      - \|H_B\|\|\Pcal_\Vcal\widetilde{g}_m\|_\LtwoSd\\
      &\,>\, m \|H\widetilde{g}_m\|_\LtwoD 
      + m \|\Pcal_\Vcal\widetilde{g}_m\|_\LtwoSd 
      - \|H_B\|\|\Pcal_\Vcal\widetilde{g}_m\|_\LtwoSd\\
      &\,\geq\, \frac{m}{2\|H \|} \|H\widetilde{g}_m\|_\LtwoD 
      + m \|\Pcal_\Vcal\widetilde{g}_m\|_\LtwoSd 
      - \|H_B\|\|\Pcal_\Vcal\widetilde{g}_m\|_\LtwoSd\\
      &\,\geq\, \frac{m}{2\|H \|} \|Hg_m\|_\LtwoD 
      + \Bigl( \frac{m}{2} - \|H_B\| \Bigr)
      \|\Pcal_\Vcal\widetilde{g}_m\|_\LtwoSd \\
      &\,\geq\, \frac{m}{2\|H \|} \|Hg_m\|_\LtwoD \,.      
    \end{split}
  \end{equation*}
  for all $m\in\N$ with $m\geq 2\|H_B\|$.
  This contradicts \eqref{eq:ProofShape1-Contradiction}, and we have
  shown part (b).
\end{proof}

\begin{remark}[Numerical implementation of
  Theorem~\ref{thm:NonlinearMonotonicity}] 
  \label{rem:FMvsMM}
  Considering for any $z\in\Rd$ a probing domain $B=B_\eps(z)$ that
  is a ball of radius $\eps>0$ around $z$, the identity
  \eqref{eq:IdentityTB} gives 
  \begin{equation*}
    \begin{split}
      \skp{P_Bg}{g}_\LtwoSd
      &\,=\, k^2 \int_{B_\eps(z)} \biggl| 
      \int_\Sd e^{\rmi k \theta\cdot z} e^{\rmi k \theta\cdot (x-z)}
      g(\theta) \ds(\theta) \biggr|^2 \dx\\
      &\,=\, k^2 |B_\eps(z)|\, \Bigl| 
      \int_\Sd e^{\rmi k \theta\cdot z} g(\theta) \ds(\theta) \Bigr|^2
      + O\bigl(k^3\eps|B_\eps(z)|\|g\|_\LtwoSd^2\bigr) \\
      &\,=\, k^2 |B_\eps(z)|\, |\skp{g}{\phi_z}_\LtwoSd|^2 
      + O\bigl(k^3\eps|B_\eps(z)|\|g\|_\LtwoSd^2\bigr) \,,
    \end{split}
  \end{equation*}
  uniformly with respect to $z\in\Rd$. 
  Here we used that $|e^{\rmi t}-1|\leq |t|$ for $t\in\R$. 
  
  If $z\in D$, then part (a) of
  Theorem~\ref{thm:NonlinearMonotonicity} implies that there is a
  finite dimensional subspace $\Vcal\tm\LtwoSd$ such that for all
  $\beta\leq\frac{\qmin}{2}$ and for all
  $g\in\Vcalperp$ with~$\|g\|_\LtwoSd = \rhotilde$,
  \begin{equation}
    \label{eq:FMvsMM1}
    \frac{\real\bigl(\skp{F(g)}{g}_\LtwoSd\bigr)}
    {\beta \skp{P_Bg}{g}_\LtwoSd}
    \,=\, \frac{\real\bigl(\skp{F(g)}{g}_\LtwoSd\bigr)}
    {\beta k^2 |B_\eps(z)| |\skp{\phi_z}{g}_\LtwoSd|^2
      + O\bigl(k^3\eps|B_\eps(z)|\|g\|_\LtwoSd^2\bigr)}
    \,\geq\, 1 \,,
  \end{equation}
  i.e.,
  \begin{equation}
    \label{eq:FMvsMM2}
    \frac{\real\bigl(\skp{F(g)}{g}_\LtwoSd\bigr)}
    {|\skp{\phi_z}{g}_\LtwoSd|^2 + O\bigl(k\eps\|g\|_\LtwoSd^2\bigr)}
    \,\geq\, k^2\beta |B_\eps| \,,
  \end{equation}
  as $\eps\to 0$.
  This shows that for any fixed $g\in\Vcalperp$ with
  $\|g\|_\LtwoSd = \rhotilde$ and $\skp{g}{\phi_z}_\LtwoSd\neq 0$ we
  can choose $\eps>0$ sufficiently small such that
  \begin{equation*}
    \frac{\real\bigl(\skp{F(g)}{g}_\LtwoSd\bigr)}
    {|\skp{\phi_z}{g}_\LtwoSd|^2}
    \,\geq\, \frac{k^2\beta |B_\eps|}{2} \,.
  \end{equation*}
  
  Similarly, if $z\notin D$, then
  part (b) of Theorem~\ref{thm:NonlinearMonotonicity} says that there
  is no finite dimensional subspace $\Wcal\tm\LtwoSd$ and no $\beta>0$
  such that \eqref{eq:FMvsMM1}--\eqref{eq:FMvsMM2} hold for all
  $g\in\Wcalperp$ with $\|g\|_\LtwoSd = \rhotilde$ as $\eps\to0$.
  
  Assuming that $\phi_z\notin\Vcal$, this says that  
  \begin{equation}
    \label{eq:MonotonicityInfCriterion}
    z\in D 
    \;\Llra\; 
    \inf \biggl\{   
    \frac{\real\bigl(\skp{F(g)}{g}_\LtwoSd\bigr)}
    {|\skp{\phi_z}{g}_\LtwoSd|^2}
    \;\bigg|\; g\in \Vcalperp \,,\,
    \|g\|_\LtwoSd=\rhotilde \,,\, 
    \skp{g}{\phi_z}_\LtwoSd\neq 0 \biggr\} 
    \,>\, 0 \,.
  \end{equation}  
  This is closely related to the inf-criterion from the nonlinear
  factorization method in \eqref{eq:NLFCharacterization}. 
  For the monotonicity criterion we have to exclude the finite
  dimensional subspace $\Vcalperp$, and we assumed that
  $\phi_z\notin\Vcal$ in the derivation of
  \eqref{eq:MonotonicityInfCriterion}, while for the factorization
  method we had to assume that $k^2$ is such that the homogeneous
  linear transmission eigenvalue problem has no nontrivial
  solution.~\hfill$\lozenge$ 
\end{remark}

In Section~\ref{sec:NumericalExamples}, we will use
\eqref{eq:MonotonicityInfCriterion} to implement the nonlinear
monotonicity based reconstruction method.
However, since the finite dimensional subspace $\Vcalperp$ that
has to be excluded is a priori unknown, we will neglect this
constraint. 
For a numerical implementation in the linear case we refer
to~\cite{GriHar18}.

\section{Numerical examples}
\label{sec:NumericalExamples}
In this section we comment on a numerical implementation of the shape
characterizations in Theorems~\ref{thm:NonlinearFactorization} and
\ref{thm:NonlinearMonotonicity}.
We consider the two-dimensional case only, i.e., $d=2$. 

Let $D\tm\Rtwo$ be open and Lipschitz bounded such that
$D\tm\BR$ for some $R>0$ sufficiently large and $\Rtwo\setminus\ol{D}$
is connected. 
We consider at third-order Kerr-type nonlinear material law that is
given by
\begin{equation}
  \label{eq:ExaContrast}
  q(x,|z|) 
  \,:=\, q_0(x) + q_1(x)|z|^2 \,,
  \qquad x\in\Rtwo \,,\; z\in\C \,,
\end{equation} 
where $q_0, q_1 \in \Linfty(\Rtwo)$ with support in $\ol{D}$ and 
$\essinf q_0>-1$.
Accordingly, the scattering problem~\eqref{eq:ScatteringProblem}
consists in determining $u=\ui+\us$ such that 
\begin{equation*}
  \Delta u + k^2 \bigl(1 + q_0 + q_1|u|^2\bigr) u
  \,=\, 0 \qquad \text{in } \Rtwo \,,
\end{equation*}
and $\us$ satisfies the Sommerfeld radiation condition.
This fits into the framework of the previous sections.

We evaluate approximate solutions of this nonlinear scattering problem 
using a fixed point iteration for the nonlinear Lippmann-Schwinger
equation 
\begin{equation*}
  \us(x) 
  \,=\, k^2 \int_D \Phi_k(x-y)
  q(x,|\ui(y)+\us(y)|) (\ui(y)+\us(y)) \dy \,, 
  \qquad x\in [-R,R]^2 \,,
\end{equation*}
as in the proof of Proposition~\ref{pro:Wellposedness}.
Denoting the solution to the linear problem by
\begin{equation}
  \label{eq:ExaFPILinearProblem}
  \us_0
    \,:=\, \bigl(I-k^2\Phi_k \ast (q_0\ph)\bigr)^{-1}
    \bigl( k^2\Phi_k\ast (q_0\ui) \bigr)
    \qquad \text{on } [-R,R]^2 \,,
\end{equation}
the fixed point iteration determines the difference $w:=\us-\us_0$. 
Starting with the initial guess~$w_0=0$ on $[-R,R]^2$ we evaluate, for 
$\ell=0,1,2,\ldots$,
  \begin{equation}
    \label{eq:ExaFPIUpdate}
    w_{\ell+1}
    \,:=\, \bigl(I-k^2\Phi_k \ast (q_0\ph)\bigr)^{-1}
    \Bigl(k^2\Phi_k\ast \bigl( q_1|w_\ell+\us_0+\ui|^2(w_\ell+\us_0+\ui)
    \bigr) \Bigr)
    \qquad \text{on } [-R,R]^2  \,.
  \end{equation}
We have seen in the proof of Proposition~\ref{pro:Wellposedness}
that this fixed point iteration converges whenever the product
$\|q_1\|_\LinftyD \|\ui\|_\LinftyD$ is sufficiently small (see
Remark~\ref{rem:Wellposedness}).  
In our numerical example below we stop the fixed point iteration when
\begin{equation}
  \label{eq:FPITolerance}
  \frac{\|w_{\ell+1}-w_{\ell}\|_{\Linfty([-R,R]^2)}}
  {\|w_{\ell+1}\|_{\Linfty([-R,R]^2)}}
  \,<\, \eps 
\end{equation}
for some tolerance $\eps>0$, and we denote the final iterate by
$w_\eps \approx w$. 
Accordingly, an approximation for the far field pattern~$\uinfty$ can
be evaluated using Proposition~\ref{pro:Farfield} by
\begin{equation*}
  \uinfty_\eps(\xhat) 
  \,=\, k^2 \int_D  \bigl( q_0(y)
  + q_1(y)|w_\eps(y)+\us_0(y)+\ui(y)|^2 \bigr) 
  \bigl(w_\eps(y)+\us_0(y)+\ui(y)\bigr)
  e^{-\rmi k \xhat\cdot y} \dy \,,
  \quad \xhat\in\Sone \,.
\end{equation*}
In \eqref{eq:ExaFPILinearProblem} and in each step of the fixed point
iteration \eqref{eq:ExaFPIUpdate} we have to solve a linear
Lippmann-Schwinger integral equation. 
For this purpose we use the simple cubature method from
\cite[Sec.~2]{Vai00}. 

Next we turn to the inverse scattering problem.  
We consider an equidistant grid of points
\begin{equation}
  \label{eq:Grid}
  \triangle
  \,=\, \{ \zij = (ih,jh) \;|\; 
  -J \leq i,j\leq J\} \,\tm\, [-R,R]^2
\end{equation}
with step size $h=R/J$ in the region of interest $[-R,R]^2$.
For each $\zij\in \triangle$ we approximate a solution to the
minimization problem
\begin{equation}
  \label{eq:MinimizationProblemFM}
  \text{Minimize } \;
  \biggl|\frac{\skp{F(g)}{g}_\LtwoSone}
  {\skp{g}{\phi_{\zij}}_\LtwoSone^2}\biggr| \;
  \text{ subject to } \;
  \|g\|_\LtwoSone=\rhotilde \;
  \text{ and } \; \skp{g}{\phi_{\zij}}_\LtwoSone\neq 0 
\end{equation}
for the nonlinear factorization method (see
Theorem~\ref{thm:NonlinearFactorization}), and 
\begin{equation}
  \label{eq:MinimizationProblemMM}
  \text{Minimize } \;
  \frac{\real\bigl(\skp{F(g)}{g}_\LtwoSone\bigr)}
  {|\skp{\phi_{\zij}}{g}_\LtwoSone|^2} \;
  \text{ subject to } \;
  \|g\|_\LtwoSone=\rhotilde \;
  \text{ and } \; \skp{g}{\phi_{\zij}}_\LtwoSone\neq 0
\end{equation}
for the nonlinear monotonicity method (see
Theorems~\ref{thm:NonlinearMonotonicity} and
Remark~\ref{rem:FMvsMM}). 

We use a composite trapezoid rule on an equidistant grid of points
\begin{equation}
  \label{eq:ObservationGrid}
  \{(\cos\phi_m,\sin\phi_m) \;|\;
  \phi_m=2\pi m/M \,,\; m=0,\ldots,M-1 \} \tm \Sone \,, \qquad M\in\N \,,
\end{equation}
to approximate the inner products in \eqref{eq:MinimizationProblemFM}
and \eqref{eq:MinimizationProblemMM}, and we discretize the densities
$g\in\Ltwo(\Sone)$ using a truncated Fourier series expansion 
\begin{equation}
  \label{eq:DFTg}
  g(\cos(t),\sin(t))
  \,=\, \sum_{n = -{N}/{2}}^{{N}/{2}-1} \ghat_n
  \frac{1}{\sqrt{2\pi}} e^{\rmi n t} \,, \qquad t\in[0,2\pi) \,,\;
  N/2\in\N \,.
\end{equation}
Accordingly, we minimize \eqref{eq:MinimizationProblemFM} and
\eqref{eq:MinimizationProblemMM} with respect to the finite
dimensional vector of Fourier coefficients
$[\ghat_{-N/2},\ldots,\ghat_{N/2-1}]^\trans\in \C^{N}$.
From our theoretical results in
Theorems~\ref{thm:NonlinearFactorization} and
\ref{thm:NonlinearMonotonicity} (see also Remark~\ref{rem:FMvsMM}), 
we expect the values of the minima in
\eqref{eq:MinimizationProblemFM} and \eqref{eq:MinimizationProblemMM}
to be close to zero when $z \in \Rtwo\setminus\ol{D}$, and
significantly larger than zero when $z\in D$.

In each grid point $\zij\in\triangle$ we approximate solutions of
\eqref{eq:MinimizationProblemFM} and \eqref{eq:MinimizationProblemMM}
using the interior point method provided by
Matlab's~\texttt{fmincon}. 
To find an appropriate initial guess $g_{ij}^{(0)}$ at each sampling
point $\zij \in \triangle$, we first perform a preliminary global
search and evaluate 
\begin{equation}
  \label{eq:StGuessFM}
  g_{ij}^{(0)}
  \,:=\, \argmin_{p,\ell,z}
  \biggl|\frac{\skp{F(g_{p,\ell,z})}{g_{p,\ell,z}}_\LtwoSone}
  {\skp{g_{p,\ell,z}}{\phi_{\zij}}_\LtwoSone^2}\biggr|
\end{equation}
for the optimization problem \eqref{eq:MinimizationProblemFM} and
\begin{equation}
  \label{eq:StGuessMM}
  g_{ij}^{(0)}
  \,:=\, \argmin_{p,\ell,z}
  \frac{\real\bigl(\skp{F(g_{p,\ell,z})}{g_{p,\ell,z}}_\LtwoSone\bigr)}
  {|\skp{\phi_{\zij}}{g_{p,\ell,z}}_\LtwoSone|^2}
\end{equation}
for the optimization problem \eqref{eq:MinimizationProblemMM}. 
Here, $g_{p,\ell,z}\in\LtwoSone$ is given by 
\begin{equation*}
  g_{p,\ell,z}(\cos(t), \sin(t))
  \,=\, \rhotilde\, \rmi^p
  \frac{1}{\sqrt{2\pi}}e^{\rmi \ell t}
  e^{-\rmi k (z_1\cos(t) + z_2\sin(t))} \,,
\end{equation*}
and the minimization in \eqref{eq:StGuessFM} and \eqref{eq:StGuessMM}
is over $p=0,1$, $\ell = -N/2,\ldots, N/2-1$, and
${z=(z_1,z_2)\in\triangle}$. 
The densities $g_{p,\ell,z}$ generate shifted Herglotz incident fields
$(Hg_\ell)(x-z)$, where~$g_\ell$ has just one active Fourier
mode. 

For each $\zij\in\triangle$ we denote the values of the final result
of the optimization by $\Ifac(\zij)$ for
\eqref{eq:MinimizationProblemFM} and $\Imon(\zij)$ for
\eqref{eq:MinimizationProblemMM}. 
Color coded plots of these indicator functions should give a
reconstruction of the support $D$ of the scattering object. 

\begin{example}
  \begin{figure}[t]
    \centering
    \includegraphics[height=5.2cm]{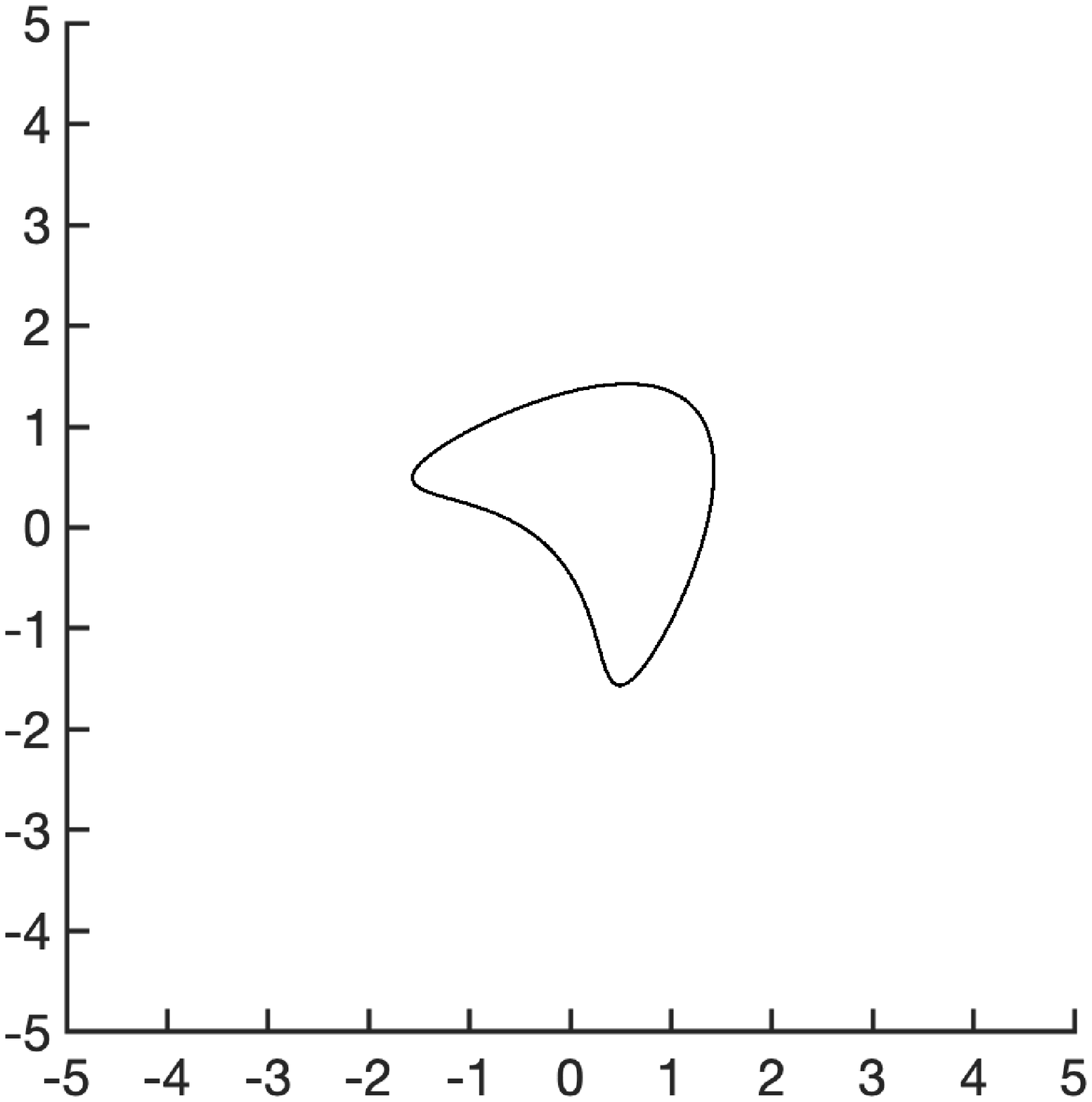}
    \includegraphics[height=5.2cm]{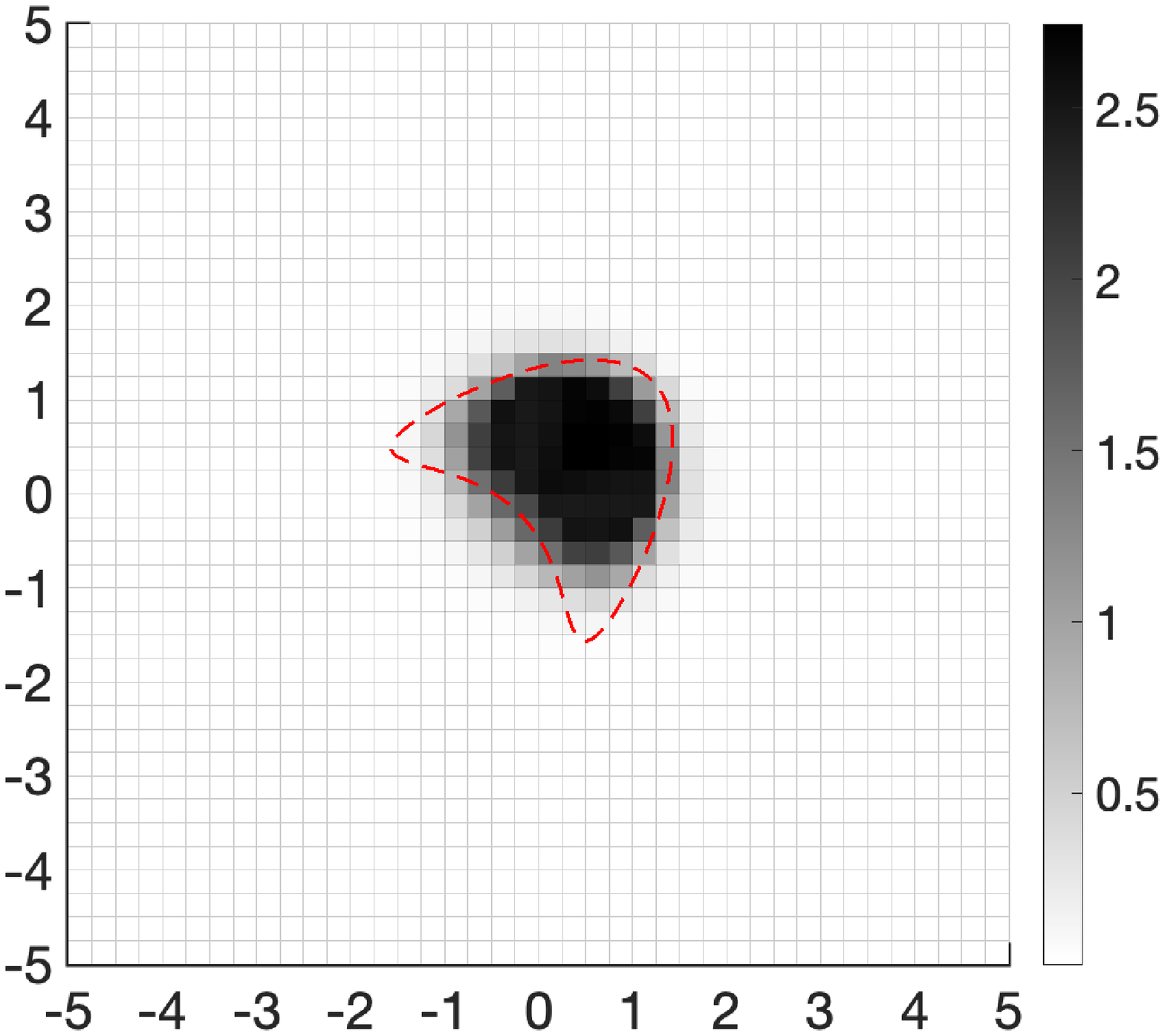}
    \includegraphics[height=5.2cm]{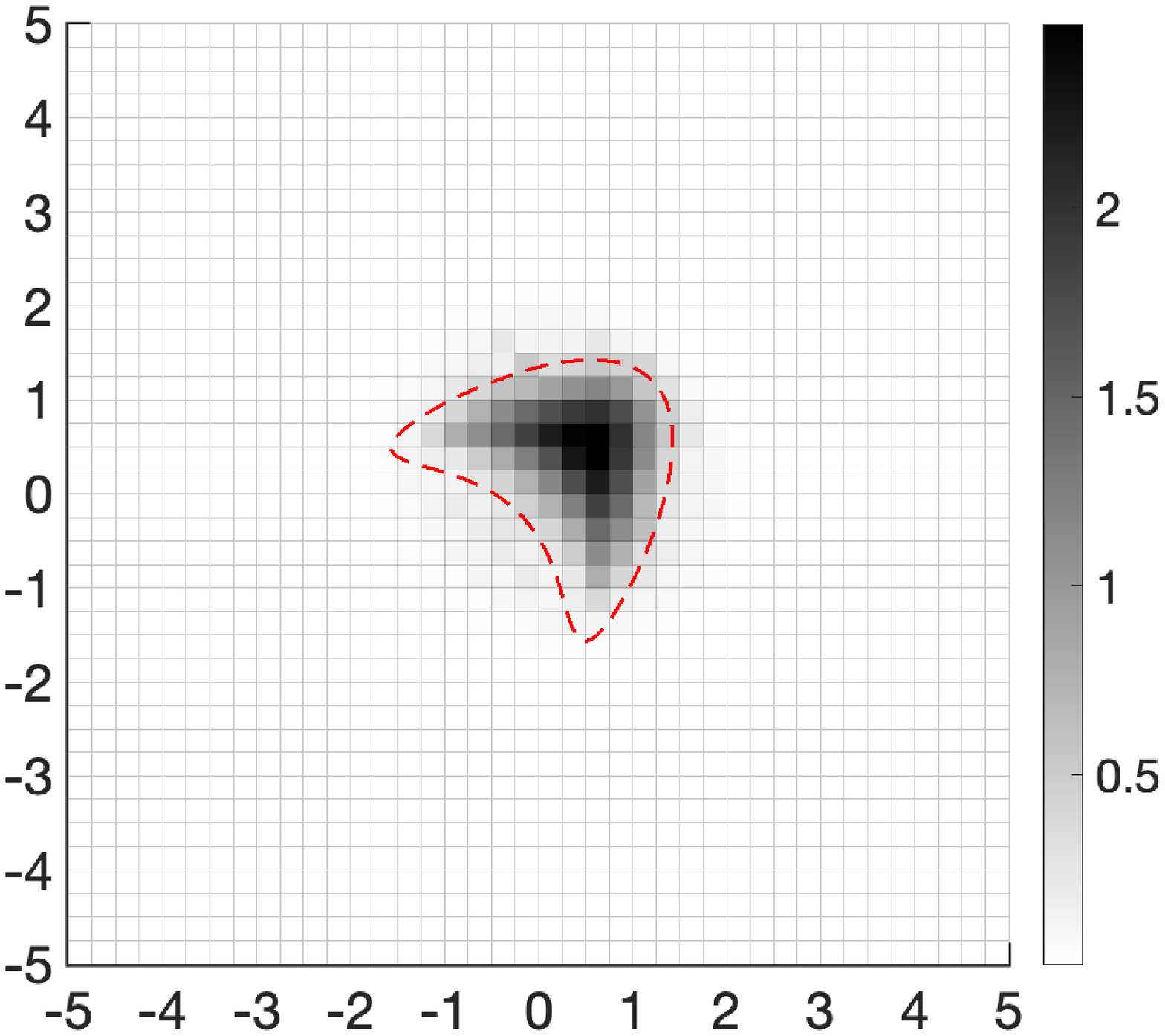}
    \caption{Nonlinear factorization method:
      Exact shape of the scattering object (left),
      initial guess $\Ifac^{(0)}$ (center), 
      final result $\Ifac$ (right).} 
    \label{fig:Exa1-1}
  \end{figure}
  
  \begin{figure}[t]
    \centering
    \includegraphics[height=5.2cm]{exa_geom}
    \includegraphics[height=5.2cm]{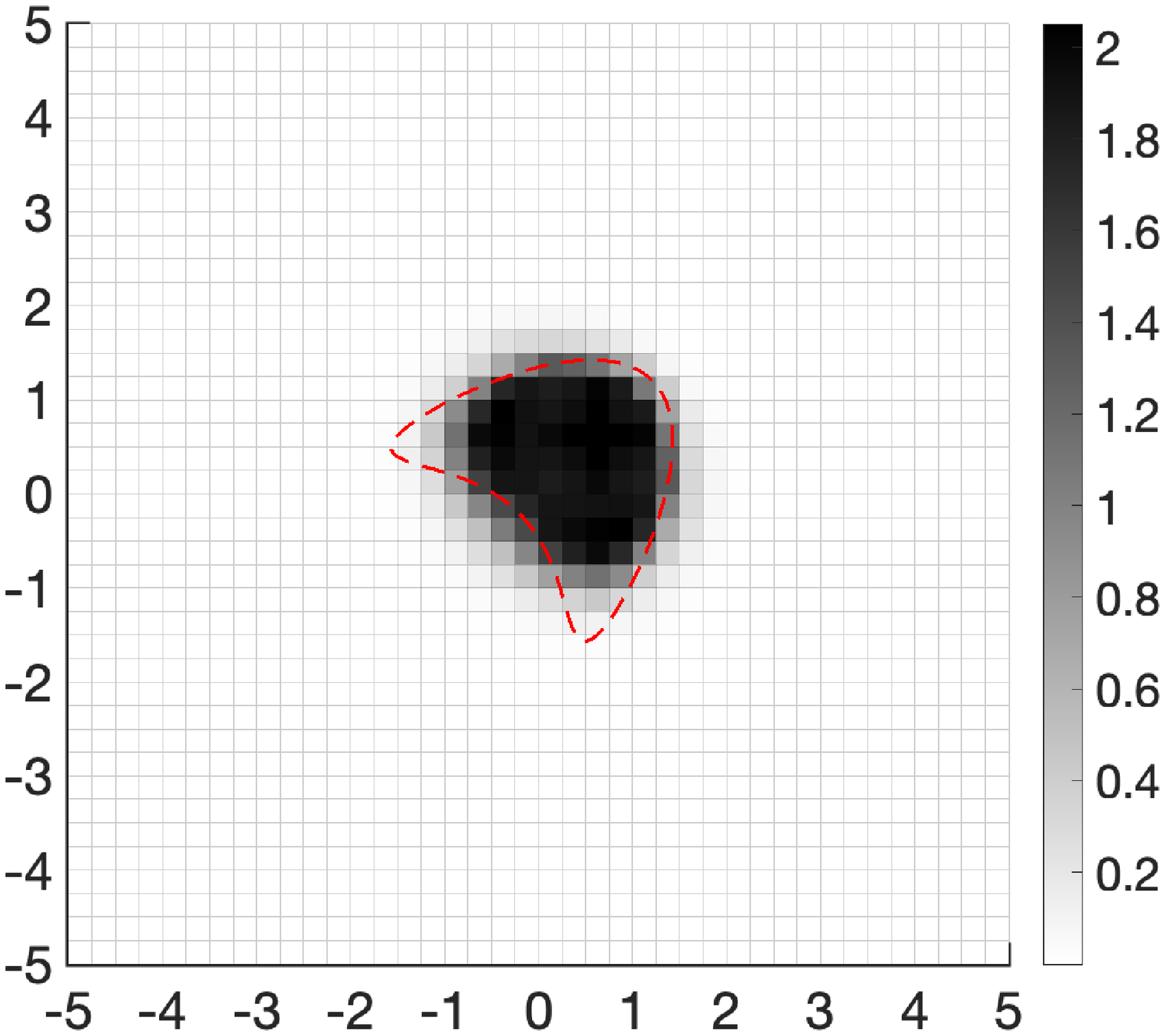}
    \includegraphics[height=5.2cm]{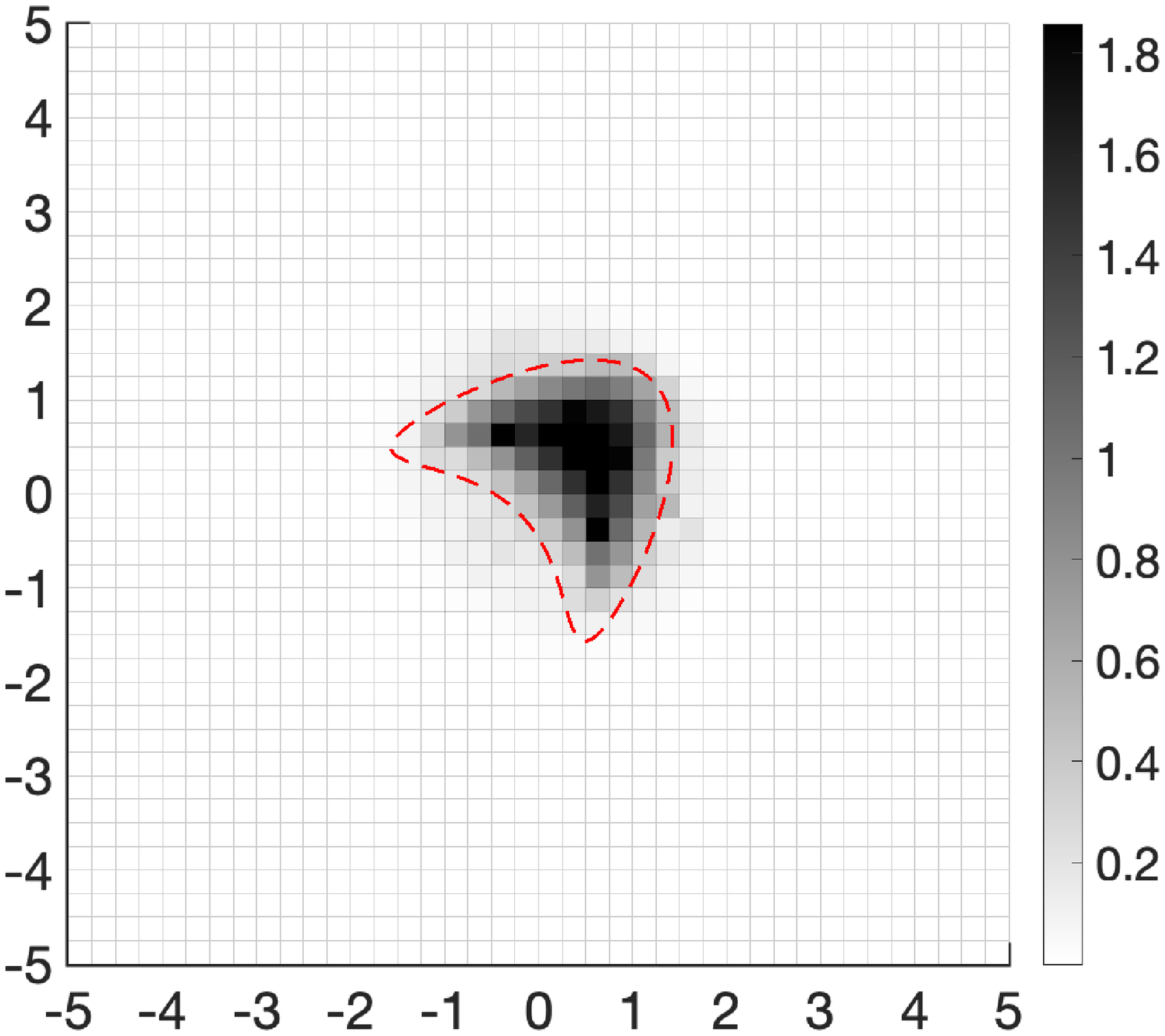}
    \caption{Nonlinear monotonicity method:
      Exact shape of the scattering object (left),
      initial guess $\Imon^{(0)}$ (center), 
      final result $\Imon$ (right).}
    \label{fig:Exa1-2} 
  \end{figure}
  
  We consider a kite shaped scattering object $D$ as shown in the left
  plot in Figure~\ref{fig:Exa1-1} and in the left plot in
  Figure~\ref{fig:Exa1-2}.  
  The coefficients in the Kerr-type nonlinear material law in
  \eqref{eq:ExaContrast} are determined to be
  \begin{equation*}
    q_0 \,=\,
    \begin{cases}
      1.16 &\text{in } D\,,\\
      0 &\text{in } \Rtwo\setminus \overline{D} \,,
    \end{cases}
    \qquad\text{and}\qquad
    q_1 \,=\,
    \begin{cases}
      2.5\cdot 10^{-22} &\text{in } D\,,\\
      0 &\text{in } \Rtwo\setminus \overline{D} \,.
    \end{cases}
  \end{equation*}
  These coefficients correspond to fused silica (see table~4.1.2 on
  p.~212 in \cite{Boy08} with $q_0 = n_0^2-1$ and $q_1 = \chi^{(3)}$). 
  For the wave number in the exterior we choose $k=1$, and the norm
  constraint in \eqref{eq:MinimizationProblemFM} and 
  \eqref{eq:MinimizationProblemMM} is set to
  $\rhotilde = 3.0\times 10^{10}$.
  
  A simple rescaling argument shows that we can equivalently work with
  \begin{equation*}
    u_{\mathrm{resc}} \,:=\, u/\tau \,,\quad
    q_{1,\mathrm{resc}} \,:=\, \tau^2 q_1 \,,\quad
    \text{and} \quad
    \rhotilde_{\mathrm{resc}} \,:=\, \rhotilde/\tau \qquad
    \text{for any } \tau>0 \,.
  \end{equation*}
  In the numerical implementation we choose
  $\tau = 3.0\times 10^{10}$, i.e.,
  $\widetilde{q_1}_{\mathrm{resc}} = 0.26 \, \eins_{D}$
  and~$\rhotilde_{\mathrm{resc}}=1$. 
  We use a sampling grid $\triangle$ as in \eqref{eq:Grid} with
  $R=5$ and $J=20$, i.e., the step size in each direction
  is~$h=0.25$. 
  Furthermore, we choose $M=256$ quadrature nodes in
  \eqref{eq:ObservationGrid}, $N=16$ Fourier modes in \eqref{eq:DFTg},
  and for the tolerance in \eqref{eq:FPITolerance} we
  choose~$\eps=10^{-5}$. 
  We compute the starting guess for the optimization
  \eqref{eq:MinimizationProblemFM} and
  \eqref{eq:MinimizationProblemMM} for each sampling point
  $\zij \in \triangle$ as in \eqref{eq:StGuessFM} or
  \eqref{eq:StGuessMM}.
  The corresponding values of the cost functional in
  \eqref{eq:MinimizationProblemFM} and
  \eqref{eq:MinimizationProblemMM} for each grid point
  $\zij\in\triangle$ are denoted by $\Ifac^{(0)}(\zij)$ and
  $\Ifac^{(0)}(\zij)$, respectively.
  Color coded plots of $\Ifac^{(0)}$ and $\Imon^{(0)}$ are shown
  in~Figure~\ref{fig:Exa1-1} (center) and Figure~\ref{fig:Exa1-2} 
  (center), respectively.
  These give already a reasonable reconstruction of the location of
  the nonlinear scattering object.
  The dashed lines indicate the exact geometry of the scatterer. 
  
  Then we approximate solutions to the optimization problems
  \eqref{eq:MinimizationProblemFM} and 
  \eqref{eq:MinimizationProblemMM} for each sampling point
  $\zij \in \triangle$ using Matlab's~\texttt{fmincon} algorithm.
  These approximations are denoted by~$\Ifac(\zij)$ and $\Imon(\zij)$, 
  respectively. 
  Color coded plots of the indicator functions $\Ifac$ and~$\Imon$ for 
  the nonlinear factorization method and for the nonlinear
  monotonicity method are shown in~Figure~\ref{fig:Exa1-1}~(right) and 
  Figure~\ref{fig:Exa1-2}~(right), respectively. 
  Again the dashed lines indicate the exact geometry of the scatterer. 
  The results obtained by the two methods are of similar quality.
  A significant improvement of the reconstruction is observed when
  compared to the initial guesses. 
  The shape of the support of the scattering object is nicely
  recovered.
  \hfill$\lozenge$
\end{example}

\section*{Conclusions}
\label{sec:Conclusions}
We have discussed a direct and inverse scattering problem for a class
of nonlinear Helmholtz equations in unbounded free space.
Assuming that the intensities of the incident waves are sufficiently
small relative to the size of the nonlinearity, we have established
existence and uniqueness of solutions to the direct and inverse
scattering problem. 
Our analysis relies on linearization techniques and estimates for the
linearization error. 
We have also considered extensions of two shape reconstruction
techniques for the inverse scattering problem, and we have provided
numerical examples.

\section*{Acknowledgments}
Funded by the Deutsche Forschungsgemeinschaft (DFG, German Research
Foundation) -- Project-ID 258734477 -- SFB 1173.

\begin{appendix}
  \section{Appendix: A useful estimate}
  In Lemma~\ref{lmm:UsefulEstimate} below we show a simple estimate
  that is used in the proof of
  Theorem~\ref{thm:UniquenessInverseProblem}, but that we have not 
  been able to find in the literature.  
  
    \begin{lemma}
    \label{lmm:UsefulEstimate}
    Let $a,b\in\C$ and $\alpha>0$.
    Then,
    \begin{equation*}
      \bigl| |a|^\alpha a - |b|^\alpha b \bigr|
      \,\leq\, 2 (|a|+|b|)^\alpha |a-b| \,. 
    \end{equation*}
  \end{lemma}

  \begin{proof}    
    Without loss of generality we can assume that $|a|\geq |b|> 0$.
    Then $t:=b/a\in\C$ satisfies~$0<|t|\leq 1$, and
    we are left to show that
    \begin{equation}
      \label{eq:Inequality-t}
      \bigl| 1-|t|^\alpha t \bigr|
      \,\leq\, 2 (1+|t|)^\alpha |1-t| \,.
    \end{equation}
    
    If $\bigl| 1-|t|^\alpha t \bigr| \leq |1-t|$ or $|t|=1$,
    then \eqref{eq:Inequality-t} is clearly satisfied. 
    Hence, we assume from now on without loss of generality that
    $\bigl| 1-|t|^\alpha t \bigr| > |1-t|$ and $0<|t|<1$.
    This implies that~$0<\real(t)\leq|t|$, and accordingly
    \begin{equation*}
      \frac{\bigl| 1-|t|^\alpha t \bigr|^2}{|1-t|^2}
      \,=\, \frac{1-2|t|^\alpha\real(t)+|t|^{2\alpha+2}}{1-2\real(t)+|t|^2}
      \,\leq\, \frac{1-2|t|^{\alpha+1}+|t|^{2\alpha+2}}{1-2|t|+|t|^2}
      \,=\, \frac{\bigl( 1-|t|^{\alpha+1} \bigr)^2}{(1-|t|)^2} \,.
    \end{equation*}
    Therefore, it suffices to show that
    \begin{equation*}
      \frac{1-|t|^{1+\alpha}}{1-|t|} 
      \,\leq\, 2 (1+|t|)^\alpha \,.
    \end{equation*} 
    
    Let $n:=\lfloor\alpha\rfloor$ and $\beta:=\alpha-n$.
    Then,
    \begin{equation*}
      (1+|t|)^n 
      \,=\, \sum_{\ell=0}^n \binom{n}{\ell} |t|^\ell
      \,\geq\, \sum_{\ell=0}^n |t|^\ell
      \,=\, \frac{1-|t|^{n+1}}{1-|t|} \,,
    \end{equation*}
    and $2(1+|t|)^\beta\geq 1 + |t|^\beta$.
    Accordingly,
    \begin{equation*}
        2 (1+|t|)^\alpha
        \,=\, \frac{1 + |t|^\beta - |t|^{n+1} -|t|^{n+\beta+1}}{1-|t|}
        \,\geq\, \frac{1-|t|^{\alpha+1}}{1-|t|} \,.
    \end{equation*}
  \end{proof}
\end{appendix}

{\small
  \bibliographystyle{abbrv}
  \bibliography{biblio}
}

\end{document}